\documentclass[11pt]{amsart}
\usepackage{amsmath, amsthm, amsfonts, hyperref, graphicx, ifpdf}
\usepackage{mathrsfs,epsfig}
\usepackage[dvipsnames,usenames]{color}
\hypersetup{
   unicode=false,          
   pdftoolbar=true,        
   pdfmenubar=true,        
   pdffitwindow=false,     
   pdfstartview={},    
   pdftitle={},    
   pdfauthor={},     
   pdfsubject={},   
   pdfcreator={},   
   pdfproducer={}, 
   pdfkeywords={}, 
   pdfnewwindow=true,      
   colorlinks=true,       
   linkcolor=blue,          
   citecolor=blue,        
   filecolor=blue,      
   urlcolor=blue          
}

\usepackage{amssymb}
\usepackage{amsmath}
\usepackage{mathrsfs}
\usepackage{verbatim}
\usepackage{tikz}
\usepackage[retainorgcmds]{IEEEtrantools}

\usepackage[margin=30 mm,heightrounded=true,centering]{geometry}
 \theoremstyle{plain}
 \newtheorem{theorem}{Theorem}[section]
 \newtheorem{lemma}[theorem]{Lemma}
 \newtheorem{question}[theorem]{Question}
 \newtheorem{corollary}[theorem]{Corollary}
 \newtheorem{proposition}[theorem]{Proposition}
 \newtheorem{conjecture}[theorem]{Conjecture}
 \newtheorem{excer}[theorem]{Excercises}
      
 \theoremstyle{definition}
 \newtheorem{definition}[theorem]{Definition}
      
 \theoremstyle{remark}
 \newtheorem{remark}[theorem]{Remark}
 
 \newcommand{\be}{\begin{equation}}
\newcommand{\ene}{\end{equation}}
\newcommand{\br}{\begin{remark}}
\newcommand{\er}{\end{remark}}
\newcommand{\bl}{\begin{lem}}
\newcommand{\el}{\end{lem}}
\newcommand{\bcor}{\begin{cor}}
\newcommand{\ecor}{\end{cor}}
\newcommand{\bpro}{\begin{pro}}
\newcommand{\epro}{\end{pro}}
\newcommand{\ben}{\begin{enumerate}}
\newcommand{\een}{\end{enumerate}}
\newcommand{\bp}{\begin{proof}}
\newcommand{\ep}{\end{proof}}
\newcommand{\bpo}{\begin{pro}}
\newcommand{\epo}{\end{pro}}
\newcommand{\beq}{\begin{equation*}}
\newcommand{\eeq}{\end{equation*}}
\newcommand{\bear}{\begin{eqnarray}}
\newcommand{\eear}{\end{eqnarray}}
\newcommand{\beqar}{\begin{eqnarray*}}
\newcommand{\eeqar}{\end{eqnarray*}}
\newcommand{\bt}{\begin{theorem}}
\newcommand{\et}{\end{theorem}}
\newcommand{\bex}{\begin{excer}}
\newcommand{\eex}{\end{excer}}

\newcommand{\R}{\mathbb{R}}

\renewcommand{\H}{\mathbb{H}}

\usepackage{syntonly}

\DeclareMathOperator{\Vol}{Vol}
\DeclareMathOperator{\Ric}{Ric}
\DeclareMathOperator{\rank}{rank}
\DeclareMathOperator{\cone}{Cone}
\DeclareMathOperator{\Fix}{Fix}

\makeatletter
\def\@setcopyright{}
\def\serieslogo@{}
\makeatother

\begin{document}
\title{On ends of finite-volume noncompact manifolds of nonpositive curvature}

\author{Ran Ji \& Yunhui Wu}
\address{Yau Mathematical Sciences Center and Department of Mathematical Sciences, Tsinghua University, Beijing 100084, P. R. China}
\email[Y.~W.]{yunhui\_wu@mail.tsinghua.edu.cn}
\address{Yau Mathematical Sciences Center, Tsinghua University, Beijing 100084, P. R. China\\
(Current) Academy for Multidisciplinary Studies, Capital Normal University, Beijing 100048, P. R. China}
\email[R.~J.]{\ jiran.jbm@gmail.com}

\begin{abstract}
In this paper we confirm a folklore conjecture which suggests that for a complete noncompact manifold $M$ of finite volume with sectional curvature $-1 \leq K \leq 0$, if the universal cover of $M$ is a visibility manifold, then the fundamental group of each end of $M$ is almost nilpotent. 
\end{abstract}

\maketitle

\section{Introduction}
The geometry and topology of complete finite-volume noncompact manifolds of nonpositive curvature is a fascinating topic for a long time. One may refer to \cite{BGS, Ball-book, Bel-survey, BH-book, Eber-book, EHS-survey} for details and related topics. In \cite{Hein76} Heintze showed that if $M$ is a complete finite-volume noncompact manifold  with pinched negative sectional curvature, then there are only finitely many distance minimizing geodesic rays emanating from a point in $M$. In particular $M$ has only finitely many ends. From the Margulis Lemma \cite[Page 101]{BGS} we know that the fundamental group of each end of $M$ is almost nilpotent, i.e., it contains a nilpotent subgroup of finite index. One may see \cite[Corollary 3.3]{Eberlein80} or \cite[Page iii]{BGS} for more details. Gromov proved in \cite{Gromov81} that if $M$ is a complete finite-volume noncompact manifold with sectional curvature $-1\leq K <0$, then $M$ is of finite type, i.e., $M$ is diffeomorphic to the interior of a compact manifold with boundary. In particular, the fundamental group of $M$ is finitely presentable.

Recall that a complete simply connected manifold $\widetilde{M}$ of nonpositive sectional curvature is called a \emph{visibility} manifold if  any two distinct points on the geometric boundary can be joined by a geodesic line. A typical example is a complete simply connected manifold of uniformly negative sectional curvature. Eberlein \cite[Theorem 3.1]{Eberlein80} showed that for a complete noncompact manifold $M$ of finite volume with sectional curvature $-1\leq K\leq 0$, if the universal cover of $M$ is a visibility manifold, then $M$ has only finitely many ends and each end is parabolic and Riemannian collared. Indeed he showed that each end is homeomorphic to a codimension-one closed aspherical submanifold, called a cross-section, cross the real line. Where each cross section is a compact quotient manifold of a horosphere \cite[Lemma 3.1g]{Eberlein80}. The following folklore conjecture has been open over decades.
\begin{conjecture}\label{Eb-C}
Let $M$ be a complete noncompact Riemannian manifold $M$ of finite volume with sectional curvature $-1\leq K\leq 0$. Suppose that the universal cover of $M$ is a visibility manifold. Then the fundamental group of each end of $M$ is almost nilpotent. 
\end{conjecture} 

\noindent This conjecture is known to be true for $\dim(M)\leq 3$ (e.g., see \cite[Corollary 15.7]{Bel-survey}). In this paper, we prove it for all dimensions.
\begin{theorem}\label{mt-1}
Conjecture \ref{Eb-C} is true. Moreover, the index can be bounded from above by a uniform constant depending only on the dimension. 
\end{theorem}

\noindent For the case that $\dim(M)=4$, we provide two different proofs to confirm Conjecture \ref{Eb-C}. 

The proof of Theorem \ref{mt-1} mainly consists of the following three parts.\\

Step-1. As introduced above, the fundamental group of each end of $M$ is isomorphic to the fundamental group of a closed aspherical manifold. In particular, the group is finitely presentable. And it is known that the fundamental group of each end of $M$ consists of parabolic isometries of the universal cover $\widetilde{M}$ of $M$ and fixes a common point in the geometric boundary of $\widetilde{M}$. A key step in the proof of Theorem \ref{mt-1} is to show that this group has subexponential growth. That is, it has vanishing algebraic entropy. This allows us to transfer Conjecture \ref{Eb-C} to a special case of a modified Milnor's problem. More precisely, a well-known Milnor's problem asks whether a finitely generated  group with vanishing algebraic entropy has at most polynomial growth. Grigorchuk in \cite{Grig-c} gave a negative answer to Milnor's problem. His counterexample is a finitely generated but not finitely presented group. He modified Milnor's problem as the following conjecture.
\begin{conjecture}[Modified Milnor Problem]\label{gap-C}
A finitely presented group with vanishing algebraic entropy has at most polynomial growth.
\end{conjecture}

\noindent A celebrated theorem of Gromov \cite{Gromov-pg} states that a finitely generated group of at most polynomial growth is almost nilpotent. One may refer to \cite{BGT-gap, CRX-gap, MR3289926, Harpe-book} for related topics on Conjecture \ref{gap-C}. Let $\varphi$ be an isometry of $\widetilde{M}$. Set $\Fix(\varphi)=\{\eta \in \widetilde{M}(\infty):\ \varphi(\eta)=\eta \}$, where $\widetilde{M}(\infty)$ is the geometric boundary of $\widetilde{M}$ (See section  \ref{sec-pre} for the precise definition). Motivated by the work \cite{Wu-para} of the second named author, in which it was shown that the translation length of a parabolic isometry of a proper visibility CAT(0) space vanishes, we show the following generalization that is the central technical proposition of the paper.
\begin{proposition} \label{0-ent-1-i}
Let $X$ be a complete proper visibility CAT(0) space and $\Gamma_0$ be a group of isometries of $X$ generated by a finite symmetric set $S$. Assume that there exists a point $\xi \in X(\infty)$ such that $\mathrm{Fix}(\varphi)=\{\xi\}$ for every $\varphi  \in \Gamma_0 \setminus \{e\}$. Then for any $x \in X$ we have
\[\lim_{k \to \infty}\frac{\max_{|\varphi |_w \leq k} d(x, \varphi(x))}{k}=0,\]
where $|\cdot|_w$ is the word norm with respect to $S$.
\end{proposition}

 Using a standard volume comparison argument, Proposition \ref{0-ent-1-i} yields the following result, which suggests that Conjecture \ref{Eb-C} is a special case of Conjecture \ref{gap-C}. 
\begin{proposition} \label{0-ent}
Let $\widetilde{M}$ be a complete, $n\ (n\geq 2)$-dimensional, simply connected, nonpositively curved visibility manifold whose Ricci curvature satisfies that $\Ric\geq -(n-1)$. And let $\Gamma_{0}$ be a finitely generated group acting freely and properly discontinuously on $\widetilde{M}$. Assume that there exists a point $\eta \in \widetilde{M}(\infty)$ such that for every $\varphi \in \Gamma_{0} \setminus \{e\}$, $\Fix(\varphi)=\{\eta\}$. Then, the algebraic entropy of $\Gamma_{0}$ vanishes.
\end{proposition}

\noindent It is known \cite[Theorem 6.5]{EO-v} that a parabolic isometry of a complete visibility manifold has a unique fixed point on its geometric boundary. Clearly Proposition \ref{0-ent} implies that the fundamental group of each end of $M$ in Conjecture \ref{Eb-C} has zero algebraic entropy. If $\dim(M)=4$, by \cite[Lemma 3.1g]{Eberlein80} it is known that the fundamental group of each end of $M$ in Conjecture \ref{Eb-C} is isomorphic to the fundamental group of a closed $3$-dimensional aspherical manifold. By using Perelman's solution of the Poincare Conjecture, Di Cerbo showed in \cite{Cerb09} that the fundamental group of a closed $3$-dimensional aspherical manifold either is almost nilpotent or has a uniform positive algebraic entropy. Thus, Conjecture \ref{Eb-C} can be confirmed for the 4-dimensional case. For general dimensions, because of the lack of information on the structures of closed aspherical manifolds, more developments are required. Together with the following two steps we will provide a proof of Theorem \ref{mt-1} for all dimensions without using Perelman's solution of the Poincare Conjecture.
\\

Step-2. Let $M$ be as in Conjecture 1.1. Denote by $\widetilde{M}$ its universal cover. The fundamental group $\Gamma$ of $M$ acts freely and properly on $\widetilde{M}$ by isometries. Let $\gamma$ be a lift of a geodesic ray in $M$ converging to an end $E$ and write $\eta=\gamma(\infty)$. The stability subgroup $\Gamma_\eta=\{\varphi \in \Gamma: \ \varphi(\eta)=\eta\}$ is isomorphic to the fundamental group of $E$. Inspired by the work of Caprace and Monod in \cite{Caprace09, CM13}, we study the action of $\Gamma_\eta$ on the transverse space $\widetilde{M}_\eta$ associated with $\eta$. The space $\widetilde{M}_\eta$ is a complete proper CAT(0) space on which $\Gamma_\eta$ acts cocompactly by isometries. However, the action of $\Gamma_\eta$ on $\widetilde{M}_\eta$ may not be discrete although $\Gamma_\eta$ is a finitely generated group. In this step we first prove Proposition \ref{BT-np}, which is a type of classical Bieberbach Theorem without the proper action assumption. Then by applying the main result of Adams-Ballmann in \cite{AB98} we will show that this transverse space $\widetilde{M}_\eta$ is a bounded space.\\

Step-3. Let $d(\cdot, \cdot)$ denote the distance function. By the definition of transverse space associated to a geometric boundary point, the boundedness of $\widetilde{M}_\eta$ in step-2 implies that for any geodesic ray $\gamma: [0,\infty)\to \widetilde{M}$ with $\gamma(\infty)=\eta$ there exists a uniform constant $C>0$ such that for every $\varphi \in \Gamma_{\eta}$, 
\begin{equation*}
d(\gamma, \varphi \circ \gamma)\leq C.
\end{equation*}
Where $d(\gamma, \varphi \circ \gamma)=\inf_{x\in \gamma([0,\infty)), y \in \varphi \circ \gamma([0,\infty))}d(x, y)$ is the distance between two geodesic rays $\gamma$ and $\varphi \circ \gamma$. Then Theorem \ref{mt-1} follows by the following version of the Margulis Lemma, which is derived in Section \ref{b-to-nil}. More precisely, we show that

\begin{proposition}[Margulis Lemma at Infinity]
\label{sec-C-Margulis}
Let $\widetilde{M}$ be a complete, simply connected $n$-dimensional manifold with sectional curvature $-1 \leq K \leq 0$ and $\Gamma_{\eta}$ be a finitely generated, discrete group of parabolic isometries that fixes some $\eta \in \widetilde{M}(\infty)$ and all horospheres centered at $\eta$. Suppose that for a geodesic ray $\gamma:[0,\infty) \to \widetilde{M}$ with $\gamma(\infty)=\eta$ there exists a uniform constant $C>0$ such that for any $\varphi \in \Gamma_{\eta}$, the distance satisfies that
\[d(\gamma, \varphi \circ \gamma)\leq C.\]   
Then $\Gamma_{\eta}$ is almost nilpotent, i.e., it contains a nilpotent subgroup of finite index. Moreover, the index is bounded from above by a uniform constant $I(n)>0$ depending only on $n$.
\end{proposition}    

\noindent In Section \ref{b-to-nil} we will provide two proofs on the first part of Proposition \ref{sec-C-Margulis}. One is to use the structure of the transverse space $\widetilde{M}_\eta$. The other one is a direct consequence of Theorem \ref{log Margulis}, a refined version of Proposition \ref{sec-C-Margulis}, which allows $C$ to go to infinity in terms of the word norm of $\varphi$.

\begin{remark}
It is not difficult to derive the existence of the uniform constant $C$ in Proposition \ref{sec-C-Margulis} under the stronger assumption that $\widetilde{M}$ is Gromov hyperbolic (see the proof of Corollary \ref{Gromov hyperbolic}). In the visibility setting it is more complicated due to the lack of uniform divergence.
\end{remark}

\begin{remark}
Caprace and Monod in \cite[Corollary E]{CM13} showed that for a proper cocompact CAT(0) space $X$, every finitely generated discrete subgroup of the isometry group of $X$ with zero algebraic entropy is almost nilpotent. Although in the proof of Theorem \ref{mt-1} we know that the transverse space $\widetilde{M}_\eta$ is a proper cocompact CAT(0) space and $\Gamma_\eta$ is a finitely generated group with zero algebraic entropy (by step-1), the group $\Gamma_\eta$ may not be a discrete subgroup of  the isometry group of $\widetilde{M}_\eta$. For example, for the case that the curvature of $M$ has pinched negative curvature, the group $\Gamma_\eta$ acts trivially on $\widetilde{M}_\eta$. Thus Theorem \ref{mt-1} is not a consequence of step-1 and \cite[Corollary E]{CM13}. Step-2 and step-3 are necessary in our proof. The second named author is grateful to Prof. Monod for explaining \cite[Corollary E]{CM13} to him.
\end{remark}

We close this introduction section by discussing several applications of Theorem \ref{mt-1}.

\subsection*{Ranks of groups}  The first application is to determine the rank of the fundamental group of each end of $M$ in Theorem \ref{mt-1}, where the rank of an almost nilpotent group is the rank of a nilpotent subgroup of finite index. Let $G$ be a finitely generated nilpotent group. Consider a filtration of $G$ as follows
\begin{equation*}
	\{1\}=G^{(0)} \triangleleft G^{(1)} \triangleleft \cdots \triangleleft G^{(d)}=G
\end{equation*}
such that each $G_{i+1} / G_i$ is an Abelian group. Define
\begin{equation*}
	\rank(G)=\sum_{i=0}^{d-1} \rank(G_{i+1} / G_i).
\end{equation*}
The rank of an almost nilpotent group is well-defined (e.g., see \cite[Page 32]{Rag72}). In  Section \ref{pofcos} we prove
\begin{corollary}\label{c-rank}
Let $M$ be a complete noncompact Riemannian manifold $M$ of finite volume with sectional curvature $-1\leq K\leq 0$. Suppose that the universal cover of $M$ is a visibility manifold. Then for each end $E$ of $M$, the rank $\rank(\pi_1(E))$ of the fundamental group of $E$ satisfies that
\[\rank(\pi_1(E))=\dim(M)-1.\]
\end{corollary}

\subsection*{Negatively curved manifolds without visibility} Theorem \ref{mt-1} gives a non-visibility criterion for complete manifolds of finite volume and bounded nonpositive curvature. In \cite{AS} Abresch and Schroeder constructed certain so-called graph manifolds $M^4$ of dimension $4$ of sectional curvature $-1\leq K<0$ and finite volume. In their construction, the fundamental group of each end of $M^4$  contains a subgroup which has exponential growth. Our second application of Theorem \ref{mt-1} is to determine the visibility of such manifolds. We prove

\begin{corollary}\label{c-nov}
Let $M^4$ be the manifold in \cite{AS} which has finite volume and sectional curvature  $-1\leq K < 0$. Then the universal cover of $M^4$ is not a visibility manifold. 
\end{corollary}   
\noindent We remark here that the universal cover of a compact non-positively curved manifold is a visibility manifold if and only if it is Gromov hyperbolic (e.g., see \cite{BH-book}). In particular, the universal cover of a compact negatively curved manifold is always a visibility manifold.   

\subsection*{Strongly asymptotic rays in ends} Recall that two geodesic rays $\gamma_1,\gamma_2$ in a complete simply connected manifold $\widetilde{M}$ of nonpositive sectional curvature  are asymptotic if $$\sup_{t \in [0,\infty) }d(\gamma_1(t),\gamma_2) < \infty.$$ It is well-known (e.g., see \cite{HH-77}) that if the sectional curvature $K$ of $\widetilde{M}$ is uniformly negative, then two asymptotic geodesic rays $\gamma_1$ and $\gamma_2$ are strongly asymptotic, i.e., $\ \lim_{t \to \infty} d(\gamma_1(t),\gamma_2)=0$. The third application of Theorem \ref{mt-1} is as follows. 
\begin{corollary}\label{c-zero}
Let $M$ be a complete noncompact Riemannian manifold $M$ of finite volume with sectional curvature $-1\leq K\leq 0$. Suppose that the universal cover $\widetilde{M}$ of $M$ is a visibility manifold. Suppose $\sigma:[0,\infty) \to M$ is a geodesic ray that converges to an end. Then any two asymptotic lifts $\gamma_1$ and $\gamma_2$ of $\sigma$ are strongly asymptotic, i.e., if $\gamma_1(\infty)=\gamma_2(\infty)$, then $$d(\gamma_1,\gamma_2)=0.$$
\end{corollary}

\noindent To our best knowledge Corollary \ref{c-zero} is new even for the case that the universal cover $\widetilde{M}$ is Gromov hyperbolic. For the proof of Corollary \ref{c-zero}, we first show Theorem \ref{mt-1}, i.e., the fundamental group of an end $E$ of $M$ is almost nilpotent. Then we apply the structure of $E$ which is homeomorphic to $C \times \R$ where $C$ is a cross-section and use an induction argument on the depth of the almost nilpotent group $\pi_1(E)$ to complete the proof.

\subsection*{One open question} In view of Theorem \ref{mt-1} and Corollary \ref{c-zero}, the following question is naturally raised.
\begin{question}\label{v-uv}
Let $(M,ds_0^2)$ be a complete noncompact Riemannian manifold of finite volume with sectional curvature $-1\leq K\leq 0$. Suppose that the universal cover $\widetilde{M}$ of $(M,ds_0^2)$ is a visibility manifold. Does $M$ admit a complete Riemannian metric $ds^2$ such that $(M,ds^2)$ has pinched negative curvature and finite volume?
\end{question}   

\noindent If the answer to Question \ref{v-uv} is positive, it is not hard to see that Theorem \ref{mt-1} is a direct consequence. In the Appendix we will construct a complete manifold $(M,ds_0^2)$ of finite volume with curvature $-1\leq K\leq 0$ such that its universal cover $(\widetilde{M},ds_0^2)$ is a visibility manifold but not a Gromov hyperbolic space. This example also tells that one cannot expect that the metric $ds^2$ in Question \ref{v-uv} is bi-Lipschitz to $ds_0^2$; otherwise, the space $(\widetilde{M},ds_0^2)$ would be Gromov hyperbolic because Gromov hyperbolicity is preserved by quasi-isometries \cite{BH-book}.

\subsection*{Plan of paper.} The paper is organized as follows. In Section \ref{sec-pre} we collect preliminaries for nonpositively curved spaces and the groups of isometries. The concept of transverse space is introduced. In Section \ref{b-to-nil} we prove an asymptotic Margulis Lemma, which deals with the case when the transverse space is bounded. A converse to the main result of Section \ref{b-to-nil} is shown in Section \ref{nil-to-b}. In Section \ref{ssec-0-ent} we show that the fundamental group of each end as stated in Theorem \ref{mt-1} has vanishing algebraic entropy. The proof of Theorem \ref{mt-1} is completed in Section \ref{sec-mt}. In Section \ref{pofcos} we prove Corollary \ref{c-rank}, \ref{c-nov} and \ref{c-zero}. Finally in the Appendix we construct a complete manifold of finite volume with bounded nonpositive curvature whose universal cover is a visibility manifold but not a Gromov hyperbolic space.

\subsection*{Acknowledgement.}
The authors would like to thank Prof. P. Eberlein and Prof. S. T. Yau for their interests. They are also grateful to anonymous referees for their careful reading and valuable comments, one of which especially improves the statement of Proposition \ref{0-ent-1-i}. The second named author is partially supported by a grant from Tsinghua university and the NSFC grants No. 12171263 and 12361141813.

\section{Preliminaries}\label{sec-pre}
In this section we set up the notations and provide necessary background on nonpositively curved geometry. One may refer to \cite{Ball-book, BGS, BH-book, Eber-book} for more details.

\subsection*{CAT(0) spaces} \hfill\\
Let $X$ be a complete geodesic metric space. $X$ is called a \emph{CAT($0$)} space if for any three points $x,y,z \in X$ we have $$d^2(z,m) \leq \dfrac{1}{2}\left(d^2 \left(z,x \right)+d^2\left(z,y\right)\right)-\dfrac{1}{4}d^2(x,y),$$ where $m$ is the midpoint of the geodesic segment from $x$ to $y$. CAT($0$) spaces are natural generalizations of complete simply connected manifolds with nonpositive sectional curvature. $X$ is said to be \emph{unbounded} if $\sup_{x,y \in X} d(x,y)=\infty$. 

We say a metric space $X$ is \emph{proper} if all closed bounded sets in $X$ are compact. We say $X$ is \emph{of bounded geometry} if for all $r>\delta>0$, there exists $n_{r,\delta}>0$ such that for any $x \in X$, there are at most $n_{r,\delta}$ disjoint geodesic balls of radius $\delta$ with centers in $B_x(r)$. A complete metric space of bounded geometry is proper.  Main examples of CAT($0$) spaces of bounded geometry include complete, simply connected Riemannian $n$-manifolds such that the sectional curvature satisfies $K \leq 0$ and the Ricci curvature satisfies $\Ric \geq -(n-1)$. One may check it directly by using the Bishop-Gromov volume comparison inequality \cite{Gromov-book}.

\subsection*{The geometric boundary} \hfill\\
Let $X$ be a CAT($0$) space. Fix a base point $x \in X$. Two geodesic rays $\gamma_1$ and $\gamma_2$ are said to be asymptotic, denoted by $\gamma_1 \sim \gamma_2$ if there exists a constant $C$ such that for any $t \geq 0$ we have
$$d(\gamma_1(t),\gamma_2(t)) \leq C.$$

Define $X(\infty)$, the \emph{geometric boundary} of $X$, to be
$$X(\infty)= \text{the set of all geodesic rays}/ \sim.$$
We denote $X \cup X(\infty)$ by $\overline{X}$. We write $\gamma(\infty)$ as the equivalence class of $\gamma$. 

When $\widetilde{M}$ is a complete, simply connected Riemannian manifold with nonpositive curvature, we fix $p \in \widetilde{M}$ and let $\mathrm{S}_p\widetilde{M}$ denote the set of unit tangent vectors in $\mathrm{T}_p \widetilde{M}$. Given $\omega \in \mathrm{S}_p \widetilde{M}$, there exists a unique geodesic ray $\gamma_\omega:[0,\infty) \to \widetilde{M}$ satisfying $\gamma(0)=p$ and $\gamma'(0)=\omega$. Two geodesic rays $\gamma_1$ and $\gamma_2$ starting from $p$ are equivalent  if and only if $\gamma_1=\gamma_2$. At the same time each equivalence class contains a representative emanating from $p$. Thus $\widetilde{M}(\infty)$ can be identified with $\mathrm{S}_p \widetilde{M}$ for each $p \in \widetilde{M}$.  A basic fact is that $\overline{\widetilde{M}}$ with the \emph{cone topology} is a compactification of $\widetilde{M}$.

\subsection*{The horocycle topology} \hfill\\
Let $\xi \in X(\infty)$. Denote by $\gamma_{x, \xi}:[0,\infty] \to \overline{X}$ the geodesic ray parametrized by arc length with $\gamma_{x, \xi}(0)=x$ and $\gamma_{x, \xi}(\infty)=\xi$. The \emph{Busemann function} $B_{\gamma_{x, \xi}}:X \to \mathbb{R}$ is defined by
\begin{equation*}
B_{\gamma_{x, \xi}}(y)=\lim_{t \to \infty}\left(d\left(\gamma_{x,\xi}\left(t\right),y\right)-t\right).
\end{equation*}
The level sets of $B_{\gamma_{x, \xi}}$ are called \emph{horospheres} centered at $\xi$ and the sublevel sets of $B_{\gamma_{x, \xi}}$ are called \emph{horoballs}. If two geodesic rays $\gamma_1 \sim \gamma_2$, then $B_{\gamma_1}-B_{\gamma_2}$ is a constant. When $\widetilde{M}$ is a complete, simply connected Riemannian manifold with nonpositive curvature, it was shown in \cite{EO-v} that there is a unique topology that makes $\overline{\widetilde{M}}$ a compactification of $\widetilde{M}$, such that for each $\eta \in \widetilde{M}(\infty)$ the set of horoballs centered at $\eta$ forms a local basis. This topology is called the \emph{horocycle topology}.

\subsection*{The visibility axiom}

\begin{definition}
A CAT($0$) space $X$ is said to satisfy the \emph{visibility axiom} if each pair of distinct points $\xi,\eta \in X(\infty)$ can be joined by a geodesic line, i.e., there is a geodesic $\gamma:\mathbb{R} \to X$ such that $\gamma(-\infty)=\xi$ and $\gamma(\infty)=\eta$. We also call such a space a \emph{visibility CAT(0) space}.
\end{definition}

\begin{remark}
From the definition it is easy to check that a complete simply connected manifold with uniformly negative sectional curvature is a visibility manifold. A more general example of a visibility manifold is a complete simply connected Gromov hyperbolic manifold of nonpositive curvature (e.g., see \cite{BH-book}).  On the other hand, visibility manifolds share a number of properties with uniformly negatively curved manifolds.
\end{remark} 

The following Lemma is a generalization of \cite[Proposition 4.8]{EO-v}.
\begin{lemma} \cite[Proposition 9.35]{BH-book}
\label{visible horo}
Let $X$ be a proper visibility CAT($0$) space. Let $\xi \in X(\infty)$ and $B_\xi$ be a horoball centered at $\xi$. Then any unbounded sequence of points in $B_\xi$  converges to $\xi$ in the cone topology. In particular for any $x\in B_\xi$, $\gamma_{x,\xi}$ is the unique geodesic ray contained in $B_\xi$.
\end{lemma}
\noindent One may also view the lemma above as follows: it is known (e.g., see \cite[part (iv) of Lemma on page 46]{BGS}) that for any $\eta\in B_\xi(\infty)\subset X(\infty)$, the Tits distance between $\eta$ and $\xi$ is no more than $\frac{\pi}{2}$. On the other hand, it is also known (e.g., see \cite[part (3) of Lemma on page 54]{BGS}) that the Tits metric on the geometric boundary of a proper visibility CAT(0) space is totally degenerated. So $\xi\in X(\infty)$ is the unique point in $B_\xi(\infty)$.   

\subsection*{Classification of Isometries} \hfill\\
Let $\varphi$ be an isometry of a CAT($0$) space $X$. The \emph{displacement function} $d_\varphi:X \to \mathbb{R}$ associated with $\varphi$ is defined by
\begin{equation*}
d_\varphi(x)=d(x,\varphi(x)).
\end{equation*}
We call
\begin{equation*}
|\varphi|=\inf_{x \in X} d_\varphi(x)
\end{equation*}
the \emph{translation length} of $\varphi$. The collection of nontrivial isometries is divided into three disjoint classes.
\begin{definition}
An isometry $\varphi:X \to X$ is called
\begin{IEEEeqnarray*}{RRL}
(\textrm{a})& & \;\; \textrm{\emph{elliptic} \;if } d_\varphi \textrm{ achieves its minimum on } X \textrm{ and } |\varphi|=0,\\
(\textrm{b})& & \;\; \textrm{\emph{axial} \;if } d_\varphi \textrm{ achieves its minimum on } X \textrm{ and } |\varphi|>0,\\
(\textrm{c})& & \;\; \textrm{\emph{parabolic} \;if } d_\varphi \textrm{ does not achieve the minimum on } X.
\end{IEEEeqnarray*}
\end{definition}

In this paper we concentrate on parabolic isometries. 

An isometry $\varphi$ of a complete, simply connected manifold $\widetilde{M}$ with nonpositive curvature extends naturally to a homeomorphism of $\overline{\widetilde{M}}$, which is still denoted by $\varphi$. By the Brouwer fixed point Theorem, $\varphi$ has at least one fixed point in $\overline{\widetilde{M}}$. It follows that every non-elliptic isometry of $\widetilde{M}$ has at least one fixed point in $\widetilde{M}(\infty)$. Denote by $\mathrm{Fix}(\varphi)$ the set of fixed points of $\varphi$ in $\overline{\widetilde{M}}$. We want to study an isometry through its action on $\widetilde{M}(\infty)$. First we have the following lemma.

\begin{lemma}
\label{parabolic}
(1). Let $\varphi$ be a parabolic isometry of a proper CAT($0$) space $X$. Then there is some $\xi \in \mathrm{Fix}(\varphi) \in X(\infty)$ such that all horospheres centered at $\xi$ are invariant under $\varphi$. 

(2). If $|\varphi|=0$, then all horospheres centered at each fixed point of $\varphi$ in $X(\infty)$ are invariant under $\varphi$. Here $\varphi$ is allowed to be elliptic.
\end{lemma}
\begin{proof}
For the proof of the first part, we refer the readers to \cite[Proposition 3.4]{Ball-book} or \cite[Part II-Lemma 8.26]{BH-book}.  

For the second part, when $|\varphi|=0$, we can find a sequence of points $\{x_i\}_{i \geq 1}$ in $X$ such that $d(x_i,\varphi(x_i))<1/i$. Let $\eta \in X(\infty)$ be a fixed point of $\varphi$ and let $H$ be a horosphere centered at $\eta$. Denote by
\begin{equation*}
B_i(x)=\lim_{t \to \infty}\left(d\left(\gamma_{x_i,\eta}\left(t\right),x\right)-t\right)
\end{equation*}
the Busemann function associated with $\gamma_{x_i,\eta}$. Since all $\gamma_{x_i,\eta}$'s are asymptotic, we have for every $i \in \mathbb{N}$,
\begin{eqnarray*}
d(H,\varphi(H))&=& |B_i(x)-B_i\left(\varphi\left(x\right)\right)|\\
&=&| \lim_{t \to \infty}\left(d\left(\gamma_{x_i,\eta}\left(t\right),x\right)-t\right)-\lim_{t \to \infty}\left(d\left(\gamma_{x_i,\eta}(t),\varphi\left(x\right)\right)-t\right)|\\
&=&| \lim_{t \to \infty}(d(\gamma_{x_i,\eta}(t),x)-d(\varphi^{-1} \circ \gamma_{x_i,\eta}(t),x))| \\
&\leq& \lim_{t \to \infty} d(\gamma_{x_i,\eta}(t),\varphi^{-1} \circ \gamma_{x_i,\eta}(t))\\
&\leq&  d(\gamma_{x_i,\eta}(0),\varphi^{-1} \circ \gamma_{x_i,\eta}(0))\\
&=& d(x_i,\varphi^{-1}(x_i)) \leq \dfrac{1}{i}.
\end{eqnarray*}
Taking the limit as $i \to \infty$ we get $d(H,\varphi(H))=0$. This completes the proof.
\end{proof}

\begin{remark} \label{re-nocpt}
We observe that if there is $\xi \in X(\infty)$ such that all horospheres centered at $\xi$ are invariant under an isometry group $\Gamma$, then the action of $\Gamma$ on $X$ cannot be cocompact. In fact, let $\gamma \in \xi$ be a geodesic ray with $\gamma(0)=x$. Then for every $\varphi \in \Gamma$ and for any $t>0$,
\begin{eqnarray*}
d(\varphi(x),\gamma(t)) &\geq& B_\gamma(\varphi(x))-B_\gamma(\gamma(t)) \\
&=& B_\gamma(x)-B_\gamma(\gamma(t))=t,
\end{eqnarray*}
which is unbounded. This means that for any $R>0$, $\bigcup_{\varphi \in \Gamma} B_{\phi(x)}(R)$ doesn't cover $X$.
\end{remark}

In the proof of Lemma \ref{parabolic} we used the fact that if two geodesic rays $\gamma_1 \sim \gamma_2$, then the function $d(\gamma_1(t),\gamma_2)$ is monotonically decreasing in $t$. We call 
\begin{equation}
\label{distance def}
d(\gamma_1,\gamma_2)=\lim_{t \to \infty} d(\gamma_1(t),\gamma_2)
\end{equation}
the \emph{asymptotic distance} between $\gamma_1$ and $\gamma_2$. We show that if $\gamma_1$ and $\gamma_2$ are parametrized such that $\gamma_1(t)$ and $\gamma_2(t)$ lie on the same horosphere for some (and hence for all) $t \in [0,\infty)$, then 
\begin{equation}
\label{two distance}
d(\gamma_1,\gamma_2) = \lim_{t \to \infty} d(\gamma_1(t),\gamma_2(t))
\end{equation}

In fact, it follows immediately from the definition \eqref{distance def} that $d(\gamma_1,\gamma_2) \leq \lim_{t \to \infty} d(\gamma_1(t),\gamma_2(t)).$ To see $d(\gamma_1,\gamma_2) \geq \lim_{t \to \infty} d(\gamma_1(t),\gamma_2(t))$, let $\gamma_2(t+t_0)$ be the projection of $\gamma_1(t)$ onto $\gamma_2$. It is implied by the convexity of horoballs  that $t_0 \geq 0$. Consider the geodesic triangle with vertices $\gamma_1(t),\gamma_1(t+t_0)$ and $\gamma_2(t+t_0)$. Since $\gamma_1$ intersects each horosphere centered at $\gamma_1(\infty)$ transversely we have $$\angle_{\gamma_1(t+t_0)}(\gamma_1(t),\gamma_2(t+t_0)) \geq \dfrac{\pi}{2}.$$ 
Since $X$ is a CAT(0) space, $$d(\gamma_1(t),\gamma_2)=d(\gamma_1(t),\gamma_2(t+t_0))\geq d(\gamma_1(t+t_0),\gamma_2(t+t_0)).$$ Taking $t \to \infty$ yields the desired inequality.

If $\varphi$ is a parabolic isometry which fixes $\xi \in X(\infty)$ and $\gamma$ is a geodesic ray with endpoint $\xi$, then $\gamma$ and $\varphi \circ \gamma$ are asymptotic. We denote simply $$d_\varphi(\gamma)=d(\gamma,\varphi \circ \gamma).$$ Moreover, if $\varphi$ also fixes all horospheres centered at $\xi$, we obtain by \eqref{two distance}
\begin{equation*}
d_\varphi(\gamma)=\lim_{t \to \infty} d_\varphi(\gamma(t)).
\end{equation*}

When $\widetilde{M}$ is a visibility manifold, the type of a non-elliptic isometry of $\widetilde{M}$ is characterized by the number of its fixed points on $\widetilde{M}(\infty)$.
\begin{lemma} \cite[Theorem 6.5]{EO-v}
\label{one fixed point}
Let $\widetilde{M}$ be a complete, simply connected, visibility manifold. Then every parabolic isometry of $\widetilde{M}$ fixes exactly one point $\eta \in \widetilde{M}(\infty)$ and hence also fixes all horospheres centered $\eta$. Every axial isometry of $\widetilde{M}$ fixes exactly two points in $\widetilde{M}(\infty)$.
\end{lemma}

The following result of the second named author will be repeatedly used in this paper.
\bt \cite[Theorem 1.3]{Wu-para}
\label{w-tl}
Let $X$ be a complete proper visibility CAT(0) space. Then, for any parabolic isometry $\varphi$ of $X$, we have
\[|\varphi|=0.\]
\et
In Section \ref{ssec-0-ent} we will see that Theorem \ref{w-tl} is a consequence of Proposition \ref{0-ent-1}.

\subsection*{Transverse spaces} \hfill\\
The study of the transverse space associated with a geometric boundary point of a CAT($0$) space was initiated by Caprace and Monod. We refer to \cite{Caprace09}, \cite{MR1934160} and \cite{CM13} for more general results and the proofs.

Let $X$ be an unbounded CAT($0$) space and $\xi \in X(\infty)$ be a boundary point. Fix a horosphere $H$ centered at $\xi$. Define $$d_\xi(x,y)=d(\gamma_{x, \xi}, \gamma_{y, \xi}) \text{ for } \ x,y \in H.$$ It is easy to see that $(H,d_\xi)$ is a pseudo-metric space. Its metric completion, denoted by $(X_\xi,d_\xi)$, is called the \emph{transverse space} of $\xi$. For example, if $X$ is a complete simply connected manifold of uniformly negative sectional curvature, then $d_\xi(x,y)\equiv 0 \text{ for all} \ x,y \in H$. That is, the space $(X_\xi,d_\xi)$ is a single point.

It turns out that not only any pair of points in $X_\xi$ can be joined by a geodesic, but also the CAT(0) structure is well-preserved (e.g., see \cite[Page 11]{MR1934160}). We will use the following property of transverse spaces, which is a reformulation of \cite[Proposition 2.8]{MR1934160} and  \cite[Proposition 3.3]{CM13}.
\begin{theorem}
\label{Transverse proper}
The space $X_\xi$ is a complete CAT($0$) space. If $X$ is of bounded geometry, then so is $X_\xi$; in particular, $X_\xi$ is proper.
\end{theorem}

Since the distance function on a CAT(0) space is convex, it is not hard to see that
\begin{equation*}\label{1-lip}
d_\xi(x,y)\leq d(x,y) \text{ for all } x,y \in H.
\end{equation*}

Let $\varphi$ be a parabolic isometry of $X$ fixing $\xi$ and all horospheres centered at $\xi$. Then $\varphi$ induces naturally an (not necessarily parabolic) isometry of $X_\xi$ with the same translation length, which is still denoted by $\varphi$ when it doesn't cause any confusion. It is straightforward to check that if a group of isometries  fixing $\xi$ acts cocompactly on a horosphere centered at $\xi$, then it also acts cocompactly on $X_\xi$.

\subsection*{Stability groups} \hfill\\
As remarked in the preceding subsections, every non-elliptic isometry of $\widetilde{M}$ fixes at least one point in $\widetilde{M}(\infty)$. Denote by $\mathrm{Isom}(\widetilde{M})$ the group of isometries of $\widetilde{M}$. Let $\Gamma$ be a subgroup of $\mathrm{Isom} (\widetilde{M})$. For $\eta \in \widetilde{M}(\infty)$, we denote by $\Gamma_\eta$ the \emph{stability subgroup} $$\Gamma_\eta=\{\varphi \in \Gamma:\ \varphi(\eta)=\eta\}$$ and by $\mathrm{Isom}_\eta (\widetilde{M})$ the stability subgroup $$\mathrm{Isom}_\eta (\widetilde{M})=(\mathrm{Isom} (\widetilde{M}))_\eta=\{\varphi \in \mathrm{Isom}(\widetilde{M}): \ \varphi(\eta)=\eta\}.$$

A group $\Gamma_0$ of isometries of $\widetilde{M}$ is called a \emph{stability group} if it fixes a point $\eta \in \widetilde{M}(\infty)$, i.e., $$\Gamma_0 \subset \mathrm{Isom}_\eta (\widetilde{M}).$$ 

For any almost nilpotent group $\Gamma'$ in $\mathrm{Isom} (\widetilde{M})$, if $\Gamma'$ acts freely on $\widetilde{M}$ and contains a parabolic isometry, it is known \cite[Lemma 7.9-(3)]{BGS} that $\Gamma'$ is necessarily a stability group.

When $\widetilde{M}$ is a visibility manifold and $\Gamma$ is a group of isometries acting freely and properly discontinuously on $\widetilde{M}$, Eberlein gave a satisfactory result on the structure of the stability subgroups of $\Gamma$. We state a reformulation of it.

\begin{theorem}\cite[Proposition 6.8]{EO-v}
\label{stability}
The non-identity elements of a stability subgroup $\Gamma_\eta$ are either all parabolic or all axial. In the axial case the group is cyclic.
\end{theorem}

We finish this section by introducing the growth of a finitely generated group.
\subsection*{Word metric and growth of groups}\hfill\\
Let $\Gamma$ be a finitely generated group with generating set $S=\{\varphi_1,\varphi_2,\cdots,\varphi_l\}$ where $l>0$ is an integer. We further assume that $S$ is symmetric, i.e., closed under the inverse operation. Recall a \emph{word} over the set $S$ is a finite sequence $\alpha=\varphi_{i_1}\varphi_{i_2} \ldots \varphi_{i_L}$ of elements of $S$. The integer $L$ is called the \emph{word length} of $\alpha$. Every word can be identified naturally with  a unique element of $\Gamma$. We remark that two different words might be identified  with the same element of $\Gamma$.

Given an element $\varphi$ of $\Gamma$, its \emph{word norm} with respect to the generating set $S$, denoted by $|\varphi|_w$, is defined to be the shortest length of a word over $S$ which can be identified with $\varphi$. The distance function $d_w$ on $\Gamma$ in the word metric with respect to $S$ is defined to be 
\begin{equation*}
d_w(\varphi_1,\varphi_2)=|\varphi^{-1}\varphi_2|_w \text{ for } \varphi_1,\varphi_2 \in \Gamma.
\end{equation*} 

The \emph{algebraic entropy} $h(\Gamma, S)$ of $\Gamma$ wit respect to $S$ is defined by
\[h(\Gamma, S)=\lim_{k \to \infty} \frac{\log{(\#B_{w}(k))}}{k},\] where $B_w(k)=\{\varphi \in \Gamma: \ |\varphi|_{w}\leq k\}$ is the ball centered at $e$ of radius $k$ in the word metric. The existence of the limit follows easily from the relation
\[B_w(k+l) \subset B_w(k) \cdot B_w(l), \quad \forall \ k, l \geq 1.\]

It is not difficult to verify that if $h(\Gamma, S)=0$, then $$h(\Gamma, S')=0$$ for every symmetric generating set $S'$ of $\Gamma$. In this case we say $\Gamma$ has \emph{subexponential growth} and write $$h(\Gamma)=0.$$

$\Gamma$ is said to have  \emph{exponential growth} if for a symmetric generating set $S$
\[h(\Gamma, S)>0.\]

It is interesting to ask when a finitely generated group with vanishing algebraic entropy is almost nilpotent. By Gromov \cite{Gromov-pg} it is known that a finitely generated group is almost nilpotent if and only if it is of polynomial growth. For this direction, one may see \cite{BGT-gap, CRX-gap, MR3289926, Harpe-book}.


\section{Boundedness implies nilpotency}\label{b-to-nil}
In this section we provide a version of the Margulis Lemma at infinity. This corresponds to Step-3 in the proof of Theorem \ref{mt-1} as introduced in the Introduction.

Let $\widetilde{M}$ be a complete, simply connected Riemannian manifold with pinched sectional curvature $-1\leq K_{\widetilde{M}} \leq 0$.  In this section we study stability groups of isometries of $\widetilde{M}$. In view of Theorem \ref{stability}, we are interested in the case when all elements are parabolic. An essential tool for the discussion in this section is the following well-known Margulis Lemma. 
\begin{theorem}[Margulis Lemma] \cite[Theorem 9.5]{BGS}
\label{Margulis-Gromov}
Let $\widetilde{M}$ be a complete simply connected $n$-dimensional Riemannian manifold with curvature $-1 \leq K_{\widetilde{M}} \leq 0$ and $\Gamma$ be a discrete group of isometries of $\widetilde{M}$. Then there exist two constants $\varepsilon=\varepsilon(n)>0$ and $I(n)\in \mathbb{N}$ such that for any $x \in M$, $\Gamma_\varepsilon(x)=\langle \varphi \in \Gamma: d_\varphi(x) \leq \varepsilon \rangle$ is an almost nilpotent group. Moreover, $N_\varepsilon(x)=\langle \varphi \in \Gamma: d_\varphi(x) \leq \varepsilon,n_\varphi(x) \leq 0.49 \rangle \subset \Gamma_\varepsilon(x)$ is a nilpotent subgroup of index $\leq I(n)$.
\end{theorem}
\noindent Where $n_\varphi(x)$, the \emph{rotational norm} of $\varphi$ at $x$, is defined to be
\begin{equation*}
n_\varphi(x)=\max \{\angle(\omega, P_{\varphi(x),x} \circ \varphi_* \omega):\ \omega \in \mathrm{S}_x \widetilde{M}\},
\end{equation*}
where $P_{\varphi(x),x}:\mathrm{T}_{\varphi(x)}\widetilde{M} \to \mathrm{T}_x \widetilde{M}$ is the parallel translation along the unique geodesic segment $\gamma_{\varphi(x),x}$ from $\varphi(x)$ to $x$, and $\varphi_*:\mathrm{T}_x \widetilde{M} \to \mathrm{T}_{\varphi(x)}\widetilde{M} $ is the pushforward associated with the map $\varphi:\widetilde{M} \to \widetilde{M}$. 

\begin{remark}
Theorem \ref{Margulis-Gromov} was generalized to complete Riemannian manifolds with low bounds on the Ricci curvature by Cheeger-Colding \cite{CC-ml} and Kapovitch-Wilking \cite{KW-ml}. Recently Breuillard-Green-Tao \cite{BGT-gap} generalizes Theorem \ref{Margulis-Gromov} to metric spaces with so-called bounded packing property. This confirms a conjecture of Gromov \cite{Gromov-book}. In this paper, we only use Theorem \ref{Margulis-Gromov}, which is the traditional Margulis Lemma .
\end{remark}
 
We are now ready to state the main result of this section as follows, which can be viewed as the Margulis Lemma at infinity.
\begin{proposition}[=Proposition \ref{sec-C-Margulis}]
\label{C Margulis}
Let $\widetilde{M}$ be a complete, simply connected $n$-dimensional manifold with sectional curvature $-1 \leq K_{\widetilde{M}} \leq 0$. Let $\Gamma_0$ be a finitely generated, discrete group of parabolic isometries that fixes some $\eta \in \widetilde{M}(\infty)$ and all horospheres centered at $\eta$. Suppose that for a geodesic ray $\gamma:[0,\infty) \to\widetilde{M}$ with $\gamma(\infty)=\eta$ there exists a uniform constant $C>0$ such that\bear \label{C bound}
d(\gamma, \varphi \circ \gamma)\leq C \text{ for all } \varphi \in \Gamma_0.
\eear
Then $\Gamma_0$ contains a nilpotent subgroup of index $\leq I(n)$, where $I(n)\in \mathbb{N}$ is as in Theorem \ref{Margulis-Gromov}.
\end{proposition}

\bp
Let $(X_{\eta}, d_{\eta})$ be the transverse space of $\eta$. Since the curvature satisfies $-1 \leq K_{\widetilde{M}} \leq 0$, by Theorem \ref{Transverse proper} we know that $(X_{\eta}, d_{\eta})$ is a complete proper CAT(0) space, on which $\Gamma_{0}$ acts by isometries. Let $x\in X_{\eta}$ correspond to the geodesic ray $\gamma:[0,\infty) \to\widetilde{M}$ in the assumption. Since $d(\gamma, \varphi \circ \gamma)\leq C$ for all $\varphi \in \Gamma_{0}$, we have
\[d_{\eta}(x,\varphi(x)) \leq C.\]  
This implies that the orbit $\{\varphi(x)\}_{\varphi \in \Gamma_{0}}$ is a bounded set in $(X_{\eta}, d_{\eta})$. By the Cartan Fixed Point Theorem (e.g., see \cite[Page 179]{BH-book}) we know that $\Gamma_{0}$ has a common fixed point, i.e., there exists a point $x_0 \in X_{\eta}$ such that 
\[\varphi(x_0)=x_0, \ \text{ for all } \varphi \in \Gamma_0.\]

\noindent Let $H$ be the horosphere centered at $\eta$ containing the point $\gamma(0)$. Recall that the transverse space $(X_\eta, d_\eta)$ of $\eta$ is the metric completion of $(H,d_\eta)$. Thus we can find a geodesic ray $\gamma_0:[0,\infty) \to\widetilde{M}$ with $\gamma_0(0)\in H$ and $\gamma_0(\infty)=\eta$ such that $$d_\eta(\gamma_0(0),x_0) \leq \varepsilon/3$$ 
where $\varepsilon>0$ is the Margulis constant as in Theorem \ref{Margulis-Gromov}. Then it follows from the triangle inequality that for every $\varphi \in \Gamma_0$,
\begin{eqnarray*}
d(\gamma_0,\varphi \circ \gamma_0)&=&d_\eta(\gamma_0(0),\varphi \circ \gamma_0(0)) \\
&\leq& d_\eta(\gamma_0(0),x_0)+d_\eta(x_0,\varphi(x_0))+d_\eta(\varphi(x_0),\varphi \circ \gamma_0(0)) \nonumber \\
&=& 2d_\eta(\gamma_0(0),x_0) \nonumber \\
&\leq& \dfrac{2\varepsilon}{3}.\nonumber
\end{eqnarray*}

\noindent Let $\{\varphi_1,\varphi_2,\cdots, \varphi_l\}$ be a finite generating set for $\Gamma_0$ where $l>0$ is an integer. The inequality above implies that for a sufficiently large constant $t>0$ one may have 
\[\max_{1\leq i \leq l}d(\gamma_0(t), \varphi_i\circ \gamma_0(t))<\varepsilon.\] 
By Theorem \ref{Margulis-Gromov} we know that the group $\Gamma_{0}=\langle\varphi_1,\varphi_2,\cdots, \varphi_l\rangle$ contains a nilpotent subgroup of index $\leq I(n)$. The proof is complete.
\ep

Proposition \ref{C Margulis} can be used to study the stability groups of isometries of a Gromov hyperbolic manifold of nonpositive curvature. Recall that, for a constant $\delta>0$, a geodesic metric space is called \emph{$\delta$-hyperbolic} if for any geodesic triangle $\triangle xyz$, the side $\gamma_{x,y}$ lies completely in the $\delta$-neighborhood of the other two sides $\gamma_{x,z}$ and $\gamma_{y,z}$. A metric space is called \emph{Gromov hyperbolic} if it is $\delta$-hyperbolic for some $\delta > 0$.

\begin{lemma}
\label{Gromov comparison}
Assume that $\widetilde{M}$ be a complete, simply connected $\delta$-hyperbolic manifold with nonpositive sectional curvature. Let $x,y,z,w$ be four distinct points  in $\widetilde{M}$ such that $d(x,y) > 2 \delta$, $\angle_x(y,w)\geq \pi/2$, $\angle_y(x,z) \geq \pi/2$ and $d(x,w)=d(y,z) \geq 2 d(x,y)$. Then we have $$d(z,w) \geq d(x,y)+2\delta.$$
\end{lemma}
\begin{proof}
Consider the geodesic triangle $\triangle xyw$. By the definition of Gromov hyperbolicity there is a point $x'$ on the geodesic segment $\gamma_{y,w}$ such that $d(x',\gamma_{x,y}) \leq \delta$ and $d(x',\gamma_{x,w}) \leq \delta$. On the other hand, in the triangle $\triangle yzw$ we know there is point $y' \in \gamma_{y,z} \cup \gamma_{z,w}$ satisfying $d(x',y') \leq \delta$. We claim that $y'$ must lie outside of $\gamma_{y,z}$; otherwise
\begin{eqnarray*}
d(x,y)&=&d(\gamma_{x,w},\gamma_{y,z})  \\
& \leq& d(x',\gamma_{x,w})+d(x',\gamma_{y,z}) \nonumber \\
&\leq& d(x',\gamma_{x,w})+d(x',y') \nonumber \\
& \leq& 2\delta, \nonumber
\end{eqnarray*}
which is a contradiction. Thus we have $y' \in \gamma_{z,w}$.

Let $w' \in \gamma_{x,y}$ be a point such that $d(x',w') \leq \delta$. Since $\angle_x(y,w),\angle_y(x,z) \geq \pi/2$, it follows from comparison with the Euclidean case that $d(w,w') \geq d(w,x)$ and $d(z,w') \geq d(z,y)$. Together with the triangle inequality we obtain
\begin{eqnarray*}
d(z,w)&=&d(z,y')+d(y',w)\\ 
&\geq& (d(z,w')-d(y',w'))+(d(w,w')-d(y',w'))\nonumber \\
&\geq& d(z,y)+d(w,x)-2(d(y',x')+d(x',w')) \nonumber \\
&\geq& 4d(x,y)-4\delta \nonumber \\
&\geq& d(x,y)+2\delta.\nonumber
\end{eqnarray*}
This completes the proof.
\end{proof}

Now we have an immediate application of Proposition \ref{C Margulis}, which generalizes a well-known result for complete manifolds with pinched negative curvature \cite[Lemma 4.9]{Bow93}. 
\begin{corollary}
\label{Gromov hyperbolic}
Let $\widetilde{M}$ be a complete, simply connected Gromov hyperbolic manifold with $-1 \leq K_{\widetilde{M}} \leq 0$.  Let $\Gamma_0$ be a finitely generated, discrete group of parabolic isometries that fixes $\eta \in \widetilde{M}(\infty)$ and hence also all horospheres centered at $\eta$. Then $\Gamma_0$ contains a nilpotent subgroup of index $\leq I(n)$, where $I(n)\in \mathbb{N}$ is as in Theorem \ref{Margulis-Gromov}.
\end{corollary}
\begin{proof}
Assume that $\widetilde{M}$ is $\delta$-hyperbolic for some $\delta>0$. Fix $p\in \widetilde{M}$ and let $\gamma_{p,\eta}:[0,\infty) \to \widetilde{M}$ be the geodesic ray with $\gamma(0)=p$ and $\gamma(\infty)=\eta$. We show that for every $\varphi \in \Gamma_0$, 
$$d(\gamma_{p,\eta},\varphi\circ \gamma_{p,\eta}) \leq 2\delta.$$

Assume not, then there exists $\varphi \in \Gamma_0$ such that for all $t>0$, $$d(\gamma_{p,\eta}(t), \varphi \circ \gamma_{p,\eta}(t)) > 2\delta.$$ Set $L=d_\varphi(p)>2\delta$. Since $\varphi$ is a parabolic isometry fixing $\eta \in \widetilde{M}(\infty)$, $\gamma_{p,\eta}(t)$ and $\varphi \circ \gamma_{p,\eta}(t)$ lie on the same horosphere with center $\eta$. It follows from the convexity of horoballs that $$\angle_{\gamma_{p,\eta}(t)}(p, \varphi \circ \gamma_{p,\eta}(t))\geq \pi/2.$$ It is then easy to check that for $t>2L$ the geodesic quadrangle with vertices $\gamma_{p,\eta}(t),\varphi \circ \gamma_{p,\eta}(t), \varphi \circ \gamma_{p,\eta}(t-2L),\gamma_{p,\eta}(t-2L)$ satisfies the assumption of  Lemma \ref{Gromov comparison} and we get
\begin{equation}
\label{Gromov quad}
d(\gamma_{p,\eta}(t-2L),\varphi \circ \gamma_{p,\eta}(t-2L)) \geq d(\gamma_{p,\eta}(t),\varphi \circ \gamma_{p,\eta}(t))+2\delta.
\end{equation}

Since $K_{\widetilde{M}} \leq 0$, the function $t \to d_\varphi(\gamma_{p,\eta}(t))$ is convex and monotonically decreasing. Therefore by \eqref{Gromov quad}, for $t>2L$,
\begin{eqnarray*}
L&=&d_\varphi(\gamma_{p,\eta}(0))\\
 &\geq& \dfrac {t}{2L} d_\varphi(\gamma_{p,\eta}(t-2L))-\dfrac{t-2L}{2L}d_\varphi(\gamma_{p,\eta}(t)) \nonumber \\
&\geq& \dfrac{\delta t}{L}+d_\varphi(\gamma_{p,\eta}(t)) \nonumber \\
&\geq& \dfrac{\delta t}{L}+2\delta,\nonumber
\end{eqnarray*}
by letting $t \to \infty$ we get a contradiction. This implies that $d_\varphi(\gamma_{p,\eta})$ is uniformly bounded for $\varphi \in \Gamma_0$ and the corollary follows from Proposition \ref{C Margulis}.
\end{proof}

It is well-known \cite[Proposition 9.32]{BH-book} that the universal cover of a compact manifold of nonpositive curvature is Gromov hyperbolic if and only if it satisfies the visibility axiom. We will give an example in the Appendix which shows that the compactness condition cannot be replaced by the finite volume condition. Hence for general visibility manifolds we need to develop new methods to derive the nilpotency of stability groups, which will be discussed in the following three sections.\\

For the rest of this section we provide a more general version of the Margulis Lemma at infinity, which replaces the constant $C$ in Proposition \ref{C Margulis} by an unbounded growth function associated with the word norm. Although we only use Proposition \ref{C Margulis} in the proof of Theorem \ref{mt-1}, the following more general result is also interesting. Namely we prove a Margulis type lemma in terms of asymptotic distance as follows.\begin{theorem}
\label{log Margulis}
Let $\widetilde{M}$ be a complete, simply connected $n$-dimensional manifold with $-1 \leq K_{\widetilde{M}} \leq 0$ and $\Gamma_0=\langle \varphi_1,\varphi_2,\cdots, \varphi_l \rangle$ be a finitely generated, discrete group of parabolic isometries that fixes $\eta \in \widetilde{M}(\infty)$ and all horospheres centered at $\eta$. Then there exists a constant $\delta>0$ such that, if for some $k \in \mathbb{N}^* \setminus \{1\}$ and for some geodesic ray $\gamma: [0,\infty) \to \widetilde{M}$ with $\gamma(\infty) \in \eta$, 
\begin{equation}
\label{log}
\max_{\varphi \in \Gamma_0, |\varphi|_w \leq k} d_\varphi(\gamma)  <  \delta\,\log k,
\end{equation}
then $\Gamma_0$ is almost nilpotent.
\end{theorem}

Clearly Proposition \ref{C Margulis} is a direct consequence of Theorem \ref{log Margulis}.

Before proving it, we provide certain necessary preparations and a useful lemma.

Let $x \in \widetilde{M}$ and $\omega \in \mathrm{T}_x \widetilde{M}$. Denote simply the holonomy $$\text{Hol}_{\varphi}(\omega)=P_{\varphi(x),x} \circ \varphi_* \omega.$$ It is easy to see that $\text{Hol}_{\varphi}$ is a linear isometry of $\mathrm{T}_x \widetilde{M}$ that preserves $\mathrm{S}_x \widetilde{M}$, hence it can thought of as an element in the orthogonal group $O(n)$. There is a natural distance function on $O(n)$ $$d(P,Q)=\max_{z \in \mathbb{R}^n} \angle(Pz,Qz),$$which makes $O(n)$ a compact manifold with the induced topology.

For $\varphi_1,\varphi_2 \in \Gamma$, define the distance between $\text{Hol}_{\varphi_1}$ and $\text{Hol}_{\varphi_2}$ based at $x$ to be the induced distance as elements of $O(n)$. To be precise, $$d_x(\text{Hol}_{\varphi_1},\text{Hol}_{\varphi_2})=\max_{\omega \in \mathrm{S}_x M} \angle(P_{\varphi_1(x),x} \circ (\varphi_1)_* \omega, P_{\varphi_2(x),x} \circ (\varphi_2)_* \omega).$$

The following lemma allows us to compare parallel translations along different piecewise geodesic paths. For completeness we provide the details for the proof.

\begin{lemma}\cite[6.2.1]{BK81}
\label{4/3}
Let $\widetilde{M}$ be a complete, simply connected Riemannian manifold with sectional curvature $-1 \leq K_{\widetilde{M}} \leq 0$. For any three distinct points $x,y,z \in M$, we have
\begin{equation*}
\max_{\omega \in \mathrm{S}_x M} \angle(P_{x,z} \omega, P_{x,y} \circ P_{y,z} \omega) \leq \dfrac{4}{3}\mathrm{Area}(\triangle xyz).
\end{equation*}

\end{lemma}
\begin{proof}
Denote by $c_0:[0,1] \to \widetilde{M}$ the geodesic segment $\gamma_{x,z}$ and by $c_1:[0,1] \to \widetilde{M}$ the piecewise geodesic segment $\gamma_{x,y} \cup \gamma_{y,z}$ satisfying $D_s \dot{c}_i(s)=0$, $i=0,1$. Let $c_t(0 \leq t \leq 1)$ be the natural homotopy from $c_0$ to $c_1$ with $D_t c_t(s)=0$ for all $s,t \in [0,1]$.

Given $\omega \in \mathrm{S}_x \widetilde{M}$, denote by $\omega_t$ the parallel transport of $\omega$ along $c_t$. We have
\begin{eqnarray*}
\angle(P_{x,z} \omega, P_{x,y} \circ P_{y,z} \omega)&=& \angle(\omega_0(1),\omega_1(1)) \\
&\leq& \int_0^1 |D_t \omega_t(1)| \,\mathrm{d} t \\
&\leq& \int_0^1 \int_0^1 |D_s D_t \omega_t(s)| \,\mathrm{d} s \mathrm{d} t \\
&=& \int_0^1 \int_0^1 |D_t D_s \omega_t(s)+R(D_t c_t,D_s c_t) \omega_t(s)| \, \mathrm{d} s  \mathrm{d} t \\
&\leq& \int_0^1 \int_0^1 \dfrac{4}{3} |D_t c_t \wedge D_s c_t| \, \mathrm{d} s  \mathrm{d} t \leq \dfrac{4}{3} \mathrm{Area}(\triangle xyz).
\end{eqnarray*}
Here we have used the purely algebraic fact that the pinched sectional curvature condition implies the boundedness of the curvature operator. Indeed by the symmetries of the curvature tensor and the Bianchi identity we have for any $p\in \widetilde{M}$ and any $\mu,\upsilon,\omega,\tau \in \mathrm{S}_p \widetilde{M}$,
\begin{eqnarray*}
&& |24 \langle R(\mu,\upsilon)\omega,\tau \rangle|\\ 
&=& |4\langle R(\mu,\upsilon+\omega)(\upsilon+\omega),\tau \rangle- 4\langle R(\upsilon,\mu+\omega)(\mu+\omega),\tau \rangle- 4\langle R(\mu,\upsilon-\omega)(\upsilon-\omega),\tau \rangle \\
&&+ 4\langle R(\upsilon,\mu-\omega)(\mu-\omega),\tau \rangle |\\
&=& |\langle R(\mu+\tau,\upsilon+\omega)(\upsilon+\omega),\mu+\tau \rangle -\langle R(\mu-\tau,\upsilon+\omega)(\upsilon+\omega),\mu-\tau \rangle \\
&& - \langle R(\upsilon+\tau,\mu+\omega)(\mu+\omega),\upsilon+\tau \rangle +\langle R(\upsilon-\tau,\mu+\omega)(\mu+\omega),\upsilon-\tau \rangle \\
&& - \langle R(\mu+\tau,\upsilon-\omega)(\upsilon-\omega),\mu+\tau \rangle + \langle R(\mu-\tau,\upsilon-\omega)(\upsilon-\omega),\mu-\tau \rangle \\
&&  +\langle R(\upsilon+\tau,\mu-\omega)(\mu-\omega),\upsilon+\tau \rangle - \langle R(\upsilon-\tau,\mu-\omega)(\mu-\omega),\upsilon-\tau \rangle |\\
&\leq& |K(\mu+\tau,\upsilon+\omega)||\mu+\tau|^2 |\upsilon+\omega|^2 +|K(\mu-\tau,\upsilon+\omega)||\mu-\tau|^2 |\upsilon+\omega|^2 \\
&& + |K(\upsilon+\tau,\mu+\omega)||\upsilon+\tau|^2 |\mu+\omega|^2 + |K(\upsilon-\tau,\mu+\omega)||\mu+\omega|^2 |\upsilon-\tau|^2\\
&& + |K(\mu+\tau,\upsilon-\omega)||\upsilon-\omega|^2 |\mu+\tau|^2 + |K(\mu-\tau,\upsilon-\omega)||\upsilon-\omega|^2|\mu-\tau|^2\\
&& + |K(\upsilon+\tau,\mu-\omega)||\mu-\omega|^2 |\upsilon+\tau|^2 + |K(\upsilon-\tau,\mu-\omega)|\mu-\omega|^2 |\upsilon-\tau|^2\\
&\leq& (|\mu+\tau|^2+|\mu-\tau|^2)|\upsilon+\omega|^2+(|\upsilon+\tau|^2+|\upsilon-\tau|^2)|\mu+\omega|^2\\
&&+ (|\mu+\tau|^2+|\mu-\tau|^2) |\upsilon-\omega|^2 + (|\upsilon+\tau|^2+ |\upsilon-\tau|^2) |\mu-\omega|^2\\
&=& 4(|\mu|^2+|\tau|^2)(|\upsilon|^2+|\omega|^2)+4(|\upsilon|^2+|\tau|^2)(|\mu|^2+|\omega|^2)=32,
\end{eqnarray*}
which implies that $|\langle R(\mu,\upsilon)\omega,\tau \rangle| \leq 4/3$. The proof of the lemma is complete.
\end{proof}

Let $\eta \in \widetilde{M}(\infty)$ be a boundary point. Assume that $\Gamma_0$ is a group of parabolic isometries fixing all horospheres centered at $\eta$. For $q \in \widetilde{M}$, denote $q_t=\gamma_{q,\eta}(t)$, where $\gamma_{q, \eta}$ is the geodesic ray emanating from $q$ to $\eta$ parametrized by arc length. Given $\varphi \in \Gamma_0$, recall that $$d_\varphi(\gamma_{q,\eta})=\lim_{t \to \infty} d(\gamma_{q,\eta}(t), \varphi \circ \gamma_{q,\eta})=\lim_{t \to \infty} d_\varphi(q_t)$$is the asymptotic distance between $\gamma_{q,\eta}$ and $\varphi \circ \gamma_{q, \eta}$.

We are now ready to prove Theorem \ref{log Margulis}.

\begin{proof}[Proof of Theorem \ref{log Margulis}]

Suppose that \eqref{log} holds for the a natural number $k \geq 2$ and for a geodesic ray $\gamma \in \eta$. Denote $p=\gamma(0)$ and $p_t=\gamma(t)$. Observe that there are only finitely many elements in $\Gamma_0$ with word norm $\leq k$. Thus we can choose $t>0$ sufficiently large so that $d_\varphi(p_t) \leq \delta \log k$ for all $\varphi \in \Gamma_0$ with $|\varphi|_w \leq k$.

Let $\varepsilon=\varepsilon(n)$ be as in Theorem \ref{Margulis-Gromov}. Because the orthogonal group $O(n)$ is compact, there exists a constant $m=m(n,\varepsilon)$ such that any collection of $m$ elements in $O(n)$ contains a pair between which the distance $\leq \theta/3$, where $\theta=0.49$ is the rotational norm bound as in Theorem \ref{Margulis-Gromov}.

We first consider the case when $k$ is sufficiently large. We claim that there exists a positive constant $K_0$ depending only on $n$, such that when $k \geq K_0$ and when $\delta$ is sufficiently small, any collection of $k$ points in the geodesic ball $B_{p_t}(\delta \log k)$ contains $m$ points lying in a geodesic ball of radius $\dfrac{\theta}{2\delta \log k}$.

In fact, let $q_1,q_2, \cdots, q_k \in B_{p_t}(\delta \log k)$ be such a collection and consider the geodesic balls $B_{q_1}(\dfrac{\theta}{2\delta \log k}), B_{q_2}(\dfrac{\theta}{2\delta \log k}), \cdots, B_{q_k}(\dfrac{\theta}{2\delta \log k})$, which are all contained in $B_{p_t}(\delta \log k+\dfrac{\theta}{2\delta \log k})$. By the Bishop volume comparison theorem, all of their volumes are bounded from below by that of the corresponding geodesic ball in $\mathbb{R}^n$. That is, 
\begin{equation}
\label{E comparison}
\mathrm{vol}(B_{q_i}(\dfrac{\theta}{2\delta \log k})) \geq \dfrac{\omega_n}{n} (\dfrac{\theta}{2\delta \log k})^n \text{ for } 1 \leq i \leq k,
\end{equation}
where $\omega_n$ denotes the $(n-1)$-dimensional area of the unit sphere in $\mathbb{R}^n$. On the other hand, the volume of $B_{p_t}(\delta \log k+\dfrac{\theta}{2\delta \log k})$ is larger than that of the corresponding geodesic ball in $\mathbb{H}^n$, which implies
\begin{eqnarray}
\label{H comparison}
\mathrm{vol}(B_{p_t}(\delta \log k+\dfrac{\theta}{2\delta \log k})) &\leq& \omega_n \int_0^{\delta \log k+\theta/(2\delta \log k)} \sinh^{n-1} r \,\mathrm{d}r \\
 &\leq& \dfrac{\omega_n}{n-1} \exp (n \delta \log k+\dfrac{n\theta}{2\delta \log k}). \nonumber
\end{eqnarray}
 
Now we let $\delta=\dfrac{1}{2n}$. If $K_0$ is sufficiently large, then
\begin{equation}
\label{volume}
k \dfrac{\omega_n}{n}(\dfrac{\theta}{2\delta \log k})^n 
 \geq m \dfrac{\omega_n}{n-1} \exp (n \delta \log k+\dfrac{n\theta}{2\delta \log k})\end{equation}
for all $k \geq K_0$. It is easy to verify that this holds, for example, when 
\begin{equation}
\label{K_0}
K_0=\exp(n^2 \varepsilon+m n^{n+1}/\varepsilon).
\end{equation}

Combining \eqref{E comparison}, \eqref{H comparison} and \eqref{volume} we obtain
\begin{eqnarray}
\label{volume cover}
\Sigma_{i=1}^k \mathrm{vol}(B_{q_i}(\dfrac{\theta}{2\delta \log k})) &\geq& k \dfrac{\omega_n}{n} (\dfrac{\theta}{2\delta \log k})^n \\
&\geq& m \dfrac{\omega_n}{n-1} \exp (n \delta \log k+\dfrac{n\theta}{2\delta \log k}) \nonumber \\
&\geq& m \mathrm{vol}(B_{p_t}(\delta \log k+\dfrac{\theta}{2\delta \log k})) \nonumber
\end{eqnarray}

In view of \eqref{volume cover}, there exists a point $q \in B_{p_t}(\delta \log k+\dfrac{\theta}{2\delta \log k})$ that is covered by at least $m$ balls among $B_{q_1}(\dfrac{\theta}{2\delta \log k})$, $B_{q_2}(\dfrac{\theta}{2\delta \log k}) \cdots, B_{q_k}(\dfrac{\theta}{2\delta \log k})$. Thus the ball centered at $q$ with radius $\dfrac{\theta}{2\delta \log k}$ contains at least $m$ points in $\{q_1,q_2, \cdots, q_k\}$ and the claim follows. 

By Theorem \ref{Margulis-Gromov}, $N=\langle \varphi \in \Gamma_0: d_\varphi(p_t) \leq \varepsilon, d_{p_t}(\text{Hol}_\varphi(p_t), \mathrm{Id}) \leq \theta=0.49 \rangle$ is a nilpotent subgroup of $\Gamma_0$. We need to show that $N$ has finite index in $\Gamma_0$.

Let $\varphi \in \Gamma_0$ be an isometry with $|\varphi|_w=j>k$. Let $\varphi_{l_1}\cdots \varphi_{l_j}$ be a shortest word over $\{\varphi_1,\varphi_2,\cdots, \varphi_l\}$ representing $\varphi$. Consider the  collection of $k$ points, $\varphi_{l_{j-k+1}}(p_t) \cdots \varphi_{l_j}(p_t)$, $\varphi_{l_{j-k+2}}(p_t) \cdots \varphi_{l_j}(p_t), \cdots, \varphi_{l_j}(p_t)$. The argument above shows that it contains $m$ points that lie in a ball of radius $\dfrac{\theta}{2\delta \log k}$, and among these $m$ points there is a pair $\alpha \beta(p_t)$ and $\beta(p_t)$ such that $|\alpha|_w \geq 1$ and $d_{p_t}(\text{Hol}_{\alpha \beta},\text{Hol}_\beta) \leq \theta/3$.  By Lemma \ref{4/3},  
\begin{eqnarray*}
n_{\beta^{-1}\alpha\beta}(p_t) &=& d_{p_t}(\text{Hol}_{\beta^{-1} \alpha \beta},\mathrm{Id})\\ 
&=& d_{\beta(p_t)}(\text{Hol}_{\alpha},\mathrm{Id}) \\
&\leq& d_{p_t}(\text{Hol}_{\alpha \beta},\text{Hol}_\beta)+\dfrac{4}{3}F,
\end{eqnarray*}
where $F$ is the area of the Euclidean comparison triangle with edge lengths $d(p_t,\beta(p_t))$, $d(p_t,\alpha \beta (p_t))$ and $d(\beta(p_t),\alpha \beta (p_t))$. We have 
\begin{eqnarray*}
F &\leq& \dfrac{1}{2} \min\{d(p_t,\beta(p_t)),d(p_t,\alpha \beta (p_t))\} \cdot d(\beta(p_t),\alpha \beta (p_t))\\
&\leq& \dfrac{1}{2}\cdot \delta \log k \cdot \dfrac{\theta}{\delta \log k}=\dfrac{\theta}{2},
\end{eqnarray*}
and it follows that $n_{\beta^{-1}\alpha\beta}(p_t) \leq \theta$. Here we have used that both $\alpha \beta$ and $\beta$ have word norm $\leq k$. This together with the fact that 
\begin{eqnarray*}
d_{\beta^{-1} \alpha \beta}(p_t)&=&d( \beta(p_t),\alpha \beta(p_t)) \\
&\leq& \dfrac{\theta}{\delta \log K_0} <\varepsilon 
\end{eqnarray*}
implies $\beta^{-1} \alpha \beta \in N$.

Therefore we can write $\varphi=\varphi' \alpha \beta$ as $\varphi=\varphi' \beta \cdot \beta^{-1} \alpha \beta$ with $\beta^{-1} \alpha \beta  \in N$. It is clear that $|\varphi'|_w \leq |\varphi|_w-1$. We can repeat the process for $\varphi'$ provided its word norm $\geq k$. Eventually we can write $\varphi$ as the product of an isometry with word norm $<k$ and an element of $N$. Since there are only finitely many isometries with word norm $<k$, $N$ has finite index in $\Gamma_0$.

The remaining case is when $k \leq K_0$. We can simply let $$\delta \leq \dfrac{\varepsilon}{\log K_0}.$$ Then if $\Gamma_0=\langle \varphi_1,\varphi_2,\cdots, \varphi_l \rangle$ satisfies \eqref{log}, it is almost nilpotent by the Margulis Lemma, since $d_{\varphi_j}(p_t) \leq \varepsilon$ for all $1 \leq j \leq l$. This together with \eqref{K_0} shows that if \eqref{log} holds for $\delta=\dfrac{\varepsilon}{n^2 \varepsilon+mn^{n+1}/\varepsilon}$, then $\Gamma_0$ is almost nilpotent.

\end{proof}

\begin{remark}
(\textrm{a}) The lower bound of the sectional curvature is essential for Proposition \ref{C Margulis} and Theorem \ref{log Margulis}. Let $\widetilde{M}$ be the warped product $\mathbb{R} \times_f \mathbb{H}^2$, where $f(t)=e^{-t}$ and $\mathbb{H}^2$ is the hyperbolic space. Let $\Gamma$ be a group of isometries acting freely, properly discontinuously and cocompactly on $\mathbb{H}^2$. It it well-known that $\Gamma$ contains a non-Abelian free subgroup, and hence cannot be almost nilpotent. However, $\Gamma$ induces naturally an isometry group of $M$ which fixes $t=\infty$. For any $q \in \mathbb{H}^2$ and every $\varphi \in \Gamma$, $d_\varphi(\gamma)=0,$ where $\gamma=(t,q),t \geq 0$ is a geodesic ray.  Let $\Pi$ be a $2$-plane in $\mathrm{T}_{(t,q)}\widetilde{M}$ spanned by the vector tangent to $\mathbb{R} \times \{q\}$ and a vector tangent to $\{t\} \times \mathbb{H}^2$. By the sectional curvature formula for warped products (\cite{BO-wp}), $K(\Pi)=-1$ and $K(\mathrm{T}_{(t,q)}(\{t\} \times \mathbb{H}^2))=-e^{2t}-1$, thus $K_{\widetilde{M}}$ is bounded from above by $-1$ and is unbounded from below.

\noindent (\textrm{b}) Consider the product space $\widetilde{M}=\mathbb{H}^2\times \R$ which is endowed with the product metric, and let $\Gamma$ be a group which is isomorphic to a punctured hyperbolic surface of finite volume. Then the curvature satisfies $-1\leq K \leq 0$ and $\Gamma \times \mathbb{Z}$ acts on $\widetilde{M}$ as isometries with a finite volume open quotient. It is not hard to see that the group $\Gamma \times \{e\}$ fixes the boundary point corresponding to the geodesic ray in the $\R$ factor. It is clear the group $\Gamma \times \{e\}$ is not almost nilpotent because $\mathbb{H}^2/\Gamma$ has finite volume. And it is not hard to see that the inequality \eqref{log} does not hold in this case. Thus, inequality \eqref{C bound} (or \eqref{log}) is also necessary for Proposition \ref{C Margulis} (or Theorem \ref{log Margulis}).
\end{remark}


\section{Nilpotency implies boundedness}\label{nil-to-b}
Let $\Gamma_0$ be a finitely generated stability group consisting of parabolic isometries which fixes $\eta \in \widetilde{M}(\infty)$ and all horospheres centered at $\eta$. In Section \ref{b-to-nil} it is proved that if there is a constant $C>0$ and a geodesic ray $\gamma:[0,\infty) \to \widetilde{M}$ with $\gamma(\infty)= \eta$ such that
\begin{equation*}
d_\varphi(\gamma) \leq C \text{ for every } \varphi \in \Gamma_{0},
\end{equation*}
then $\Gamma_0$ is almost nilpotent.

The aim of this section is to show the converse under the additional assumption that $\Gamma_0$ acts cocompactly on a horosphere centered at $\eta$ and under a relaxed curvature condition without lower bound. Moreover, we show that the  constant $C$ in the inequality above is equal to $0$ in this case. More precisely,
\bt \label{nil-to-0}
Let $\widetilde{M}$ be a complete, simply connected visibility manifold with sectional curvature $ K_{\widetilde{M}} \leq 0$. Let $\eta \in \widetilde{M}(\infty)$ be a boundary point. Assume that there is an almost nilpotent stability group $\Gamma_0 \subset \mathrm{Isom}_\eta \widetilde{M}$ consisting of parabolic isometries which acts freely, properly discontinuously and cocompactly on every horosphere centered at $\eta$. Then $\eta$ is a zero point with respect to $\Gamma_0$, i.e., 
\begin{equation*}\label{adist=0}
d(\gamma,\varphi \circ \gamma)=0
\end{equation*}
for any geodesic ray $\gamma \in \eta$ and for every $\varphi \in \Gamma_0$.
\et

\begin{proof} [Proof of Theorem \ref{nil-to-0}]
We first prove the case that $\Gamma_0$ itself is nilpotent. By definition there is a finite series of normal subgroup 
\begin{equation}
\label{def nilpotent}
\{1\}=\Gamma_0^{(0)} \triangleleft \Gamma_0^{(1)} \triangleleft \cdots \triangleleft \Gamma_0^{(d)}=\Gamma_0,
\end{equation}
where $\Gamma_0^{(i)}=[\Gamma_0^{(i+1)},\Gamma_0]$, $i=0,1,\cdots,d-1$.

Let $H_0$ be a horosphere centered $\eta$. Since $H_0/\Gamma_0$ is compact, there is a closed bounded fundamental domain $F_0 \subset H_0$ for $\Gamma_0$. It is easy to check that for any $t \geq 0$, $F_{-t}=\{q_t=\gamma_{q, \eta}(t): \ \gamma_{q, \eta}(0)= q \in F_0\}$ is a fundamental domain for $\Gamma_0$ in the horosphere $H_{-t}=\{q_t=\gamma_{q, \eta}(t): \ \gamma_{q, \eta}(0)= q \in H_0 \}$. Set $$D=\max_{q,r \in F_0} d(q,r)$$to be the diameter of $F_0$. 

For $\varphi \in \Gamma_0$, define a continuous function by
\begin{eqnarray*}
D_\varphi: \widetilde{M} &\to& \mathbb{R}^{\geq 0} \\
q &\mapsto& d_\varphi (\gamma_{q, \eta})= \lim_{t \to \infty} d(\gamma_{q, \eta} (t), \varphi \circ \gamma_{q,\eta} (t)).
\end{eqnarray*}
Since $\varphi$ is a parabolic isometry of $\widetilde{M}$, this function $D_\varphi$ is well-defined. From the definition we have for any $t \geq 0$, $$D_\varphi(q)=D_\varphi(\gamma_{q, \eta}(t)).$$ 

When $\varphi \in \Gamma_0^{(1)}$, which is contained in the center of $\Gamma_0$, we have for every $\psi \in \Gamma_0$ and any geodesic ray $\gamma:[0,\infty) \to \widetilde{M}$ with $\gamma(0)=q$ and $\gamma(\infty)=\eta$,
\begin{eqnarray*}
D_\varphi(\psi(q))&=&d(\psi \circ \gamma, \varphi \psi \circ \gamma)\\
&=&d(\gamma,\psi^{-1} \varphi \psi \circ \gamma)\\
&=&d(\gamma,\psi^{-1}\psi \cdot \varphi \circ \gamma) \\
&=&d(\gamma,\varphi \circ \gamma)\\
&=&D_\varphi(q).
\end{eqnarray*}
That is, $D_\varphi$ is invariant under the action of $\Gamma_0$. By the continuity of $D_\varphi$ we can find $p_0,q_0 \in F_0$ so that $$D_\varphi(p_0)=\min_{p \in \widetilde{M}} D_\varphi(p)=\underline{k}, \;\;\;\; D_\varphi(q_0)=\max_{p \in \widetilde{M}} D_\varphi(p)=\overline{k}.$$

We claim that $\underline{k}=\overline{k}$. 

Otherwise, let $\delta=\dfrac{\overline{k}-\underline{k}}{2}>0$. We will arrive at a contradiction.

By the definition of $D_\varphi$ there exists $T>0$ such that 
\begin{equation}
\label{choice 1}
d_\varphi(\gamma_{p_0, \eta}(t))<\underline{k}+\delta \text{ for all } t\geq T 
\end{equation}
and 
\begin{equation}
\label{choice 2}
d_\varphi(r)<\overline{k}+\delta \text{ for all } r \in \bigcup_{t \geq T} H_{-t}.
\end{equation}

Fix a constant $S>T+D$, where $D$ is the diameter of $F_0$. Let $r_0 \in \widetilde{M}$ be the unique point such that $\gamma_{q_0, \eta}(S)$ is the midpoint of the geodesic segment from $\gamma_{p_0, \eta}(S)$ to $r_0$. Observe that $$d(r_0, \gamma_{q_0, \eta}(S))=d(\gamma_{p_0, \eta}(S),\gamma_{q_0, \eta}(S)) \leq D$$ and $$\gamma_{q_0, \eta}(S) \in H_{-S} \subset \bigcup_{t \geq T+D} H_{-t},$$ so we have $$r_0 \in \bigcup_{t \geq T}H_{-t}.$$ It then follows from \eqref{choice 1} \eqref{choice 2} and the convexity of $d_\varphi$ that
\begin{eqnarray*}
\overline{k} &\leq& d_\varphi(\gamma_{q_0, \eta}(S)) \\
&\leq& \dfrac{1}{2} (d_\varphi(\gamma_{p_0, \eta}(S))+d_\varphi(r_0)) \nonumber\\
&<& \dfrac{1}{2}(\overline{k} +\underline{k})+ \delta \nonumber
\end{eqnarray*}
which contradicts the choice of $\delta$. Thus $$\underline{k}=\overline{k}.$$ 

On the other hand, since $D_\varphi(q) \leq d_\varphi(q)$, $\underline{k}=\min D_\varphi \leq \inf d_\varphi=|\varphi|$. It is implied by Theorem \ref{w-tl} and the visibility of $\widetilde{M}$ that $|\varphi|=0$. Hence $\overline{k}=\underline{k}=0$, that is,  
\begin{equation*}
d_\varphi(\gamma)=0 \text{ for all } \varphi \in \Gamma_0^{(1)}  \text{ and any geodesic ray }\gamma \in \eta.
\end{equation*}

Proceeding by induction, suppose that $d_\varphi(\gamma)=0$ for all $\varphi \in \Gamma_0^{(i)}$ and any geodesic ray $\gamma$ with $\gamma(\infty)=\eta$. 

Now for $\varphi \in \Gamma_0^{(i+1)}$, by \eqref{def nilpotent} for every $\psi \in \Gamma_0$ the commutator $\alpha=[\psi^{-1},\varphi] \in \Gamma_0^{(i)}$. Let $\gamma$ be a geodesic ray with $\gamma(0)=q$ and $\gamma(\infty)=\eta$, we have
\begin{eqnarray*}
D_\varphi(\psi(q))&=&d(\psi \circ\gamma,\varphi \psi \circ \gamma)\\
&=&d(\gamma, \psi^{-1} \varphi \psi \varphi^{-1}  \cdot \varphi \circ \gamma)  \\
&=&d_{\alpha \varphi}(\gamma).
\end{eqnarray*}

By the  inductive hypothesis, $D_\alpha \equiv 0$ on $\widetilde{M}$. It then follows from the triangle inequality that
\begin{eqnarray*}
|D_\varphi(\psi(q))-D_\varphi(q)|&=&|d_{\alpha  \cdot \varphi}(\gamma)-d_\varphi(\gamma)| \\
&\leq& |d_\alpha(\varphi \circ \gamma)| \\
&=& |D_\alpha(\varphi(q))|=0,
\end{eqnarray*}
thus $D_\varphi$ is invariant under $\Gamma_0$. A straightforward adaptation of the argument for $\Gamma_0^{(1)}$ shows that 
\begin{equation*}
\min_{p \in \widetilde{M}} D_\varphi(p)=\max_{p \in \widetilde{M}} D_\varphi(p)=0 \text{ for all } \varphi \in \Gamma_0^{(i+1)}.
\end{equation*}
The proof of the case when $\Gamma_0$ is nilpotent is complete by induction.

Now consider the case when $\Gamma_0$ is almost nilpotent. By definition we can write $\Gamma_0=N \varphi_1 \cup \cdots \cup N \varphi_k$ as a finite union of right cosets of $N$, where $N \subset \Gamma_0$ is a nilpotent subgroup of finite index. Let $F \subset H_0$ be a fundamental domain for $\Gamma_0$, then it is straightforward to check that $\varphi_1(F) \cup \cdots \cup \varphi_k(F)$ is a closed bounded fundamental domain for $N$ in $H_0$. It follows that $N$ also acts cocompactly on $H_0$. We have already shown that $d_\varphi(\gamma)=0$ for every $\varphi \in N$ and for any geodesic ray $\gamma$ with $\gamma(\infty)=\eta$. As for $\varphi_j$, $1 \leq j \leq k$, we have for every $\varphi \in N$ and for any geodesic ray $\gamma \in \eta$,
\begin{eqnarray*}
|d_{\varphi_j}(\gamma)-d_{\varphi_j}(\varphi \circ \gamma)|&=& |d(\gamma,\varphi_j \circ \gamma)-d(\varphi \circ \gamma, \varphi_j \varphi \circ \gamma)|\\
&\leq& d(\gamma, \varphi \circ \gamma)+d(\varphi_j \circ \gamma, \varphi_j  \varphi \circ \gamma)\\
&\leq& 2|d_\varphi(\gamma)|=0 
\end{eqnarray*}
by the triangle inequality. Equivalently, $d_{\varphi_j}$ is invariant under $N$. Repeating the preceding argument, we obtain that the minimum and maximum of $d_{\varphi_j}$ have equal value $0$. Now for a general element $\varphi  \varphi_j  \in \Gamma_0$, where $\varphi \in N$, it follows that
\begin{eqnarray*}
|d_{\varphi  \varphi_j}(\gamma)|&\leq& |d(\gamma, \varphi_j \circ \gamma)+d(\varphi_j \circ \gamma, \varphi  \varphi_j \circ \gamma)|\\
&\leq& d_{\varphi_j}(\gamma)+d_{\varphi}(\varphi_j \circ \gamma)\\
&=&0.
\end{eqnarray*}
This completes the proof of Theorem \ref{nil-to-0}.
\end{proof}
\begin{remark}
In the proof of Theorem \ref{nil-to-0} above, using Theorem \ref{w-tl} is the unique place we apply the visibility assumption.
\end{remark}

Let $M$ be as in Conjecture \ref{Eb-C}. Let $\sigma:[0,\infty) \to M$ be a geodesic ray converging to an end $E$. Take a lift $\gamma$ of $\sigma$ in the universal cover $\widetilde{M}$ of $M$. By \cite[Lemma 3.1g]{Eberlein80} the fundamental group $\pi_1(E)$ acts discretely and cocompactly on all horospheres in $\widetilde{M}$ centered at $\gamma(\infty)$. In view of Proposition \ref{C Margulis} and Theorem \ref{nil-to-0}, the following three assertions are equivalent:
\begin{IEEEeqnarray*}{RRL}
(\textrm{a})& & \;\; d(\gamma, \varphi \circ \gamma) \leq C \text{ for all } \varphi \in \pi_1(E) \text{ and for some constant } C\geq 0.\\
(\textrm{b})& & \;\; \pi_1(E)\text{ is almost nilpotent}. \\
(\textrm{c})& & \;\; d(\gamma, \varphi \circ \gamma) = 0 \text{ for all } \varphi \in \pi_1(E).
\end{IEEEeqnarray*}

Therefore, to prove Theorem \ref{mt-1}, it is sufficient to show (\textrm{a}). In the following two sections we will prove this inequality.

\section{Zero algebraic entropy}\label{ssec-0-ent}

Let $M$ denote a complete noncompact manifold of finite volume and nonpositive sectional curvature $-1 \leq K_M \leq 0$. We further assume that the universal cover of $M$ satisfies the visibility axiom.  The main result of this section is the following theorem, which suggests that Conjecture \ref{Eb-C} is a special case of Conjecture \ref{gap-C}.
\bt \label{0-ent-v}
For each end $E \subset M$, the fundamental group of $E$ has zero algebraic entropy, i.e., 
\[h(\pi_1(E))=0.\]
\et

By Eberlein's work \cite[Theorem 3.1]{Eberlein80} it is known that the fundamental group of an end $E$ consists of parabolic isometries of the universal cover $\widetilde{M}$ of $M$ and fixes a common boundary point in $\widetilde{M}(\infty)$. Since $\widetilde{M}$ is a visibility manifold, by Lemma \ref{one fixed point} we know  that there exists $\eta \in \widetilde{M}(\infty)$ such that the set of fixed points $\mathrm{Fix}(\varphi)=\{\eta\}$ for every $\varphi \in \pi_1(E) \setminus \{e\}$.  Actually Theorem \ref{0-ent-v} is a consequence of the following more general result.

\begin{proposition}[=Proposition \ref{0-ent}] \label{sec-0-ent}
Let $\widetilde{M}$ be a complete, $n\ (n \geq 2)$-dimensional, simply connected, nonpositively curved visibility manifold whose Ricci curvature satisfies that $\Ric_{\widetilde{M}}\geq -(n-1)$. Let $\Gamma_{0}$ be a finitely generated group acting freely and properly discontinuously on $\widetilde{M}$. Assume that there exists a point $\eta \in \widetilde{M}(\infty)$ such that $\mathrm{Fix}(\varphi)=\{\eta\}$ for every $\varphi  \in \Gamma_{0} \setminus \{e\}$. Then, the algebraic entropy of $\Gamma_{0}$ vanishes, i.e.,
\[h(\Gamma_{0})=0.\]
\end{proposition}

We split the proof of Proposition \ref{sec-0-ent} into two parts. The first part holds in the more general setting of geodesic metric spaces. More precisely,
\begin{proposition}[= Proposition \ref{0-ent-1-i}] \label{0-ent-1}
Let $X$ be a complete proper visibility CAT(0) space and $\Gamma_0$ be a group of isometries of $X$ generated by a finite symmetric set $S$. Assume that there exists a point $\xi \in X(\infty)$ such that $\mathrm{Fix}(\varphi)=\{\xi\}$ for every $\varphi  \in \Gamma_0 \setminus \{e\}$. Then for any $x \in X$ we have
\[\lim_{k \to \infty}\frac{\max_{|\varphi |_w \leq k} d(x, \varphi(x))}{k}=0,\]
where $|\cdot|_w$ is the word norm with respect to $S$.
\end{proposition}

\begin{remark}
The limit $\lim_{k \to \infty} \max_{|\varphi |_w \leq k} d(x, \varphi(x))/k$ is called the \emph{asymptotic joint displacement} of $S$ by Breuillard and Fujiwara in \cite{MR4275871}. They study the relation between the asymptotic joint displacement and other quantities that arise from geometric group theory for general CAT($0$) spaces and Gromov hyperbolic spaces.
\end{remark}

Before the proof, we point out that  Theorem \ref{w-tl} of the second named author is a direct consequence of Proposition \ref{0-ent-1}. Recall that Theorem \ref{w-tl} states that the translation length of each parabolic isometry of a complete proper visibility CAT($0$) is $0$.

\bp [Proof of Theorem \ref{w-tl}]
Since $X$ is a complete proper visibility CAT($0$) space and $\varphi$ is parabolic, it is known (see the proof of \cite[Lemma 6.8]{BGS}) that $\varphi$ has a unique fixed point $\xi \in X(\infty)$. Moreover, $\mathrm{Fix}(\varphi^k)=\xi$ for every $k \in \mathbb{Z} \setminus \{0\}$. Recall that (see \cite[Lemma 6.6]{BGS}) 
\[|\varphi|=\lim_{k \to \infty}\frac{d(x,\varphi^k(x))}{k}\]
is independent of the choice of $x \in X$. Then the conclusion follows from Proposition \ref{0-ent-1} by choosing $\Gamma_0=\langle \varphi \rangle$ that is the cyclic group generated by $\varphi$.
\ep

The proof of Proposition \ref{0-ent-1} is motivated by the method of Karlsson-Margulis in \cite{KM}. Before providing details, we list several facts that will be applied later.

Let $$\gamma:[0, \infty) \to X$$ be the unique geodesic ray satisfying that $\gamma(0)=x$ and $\gamma(\infty)=\xi$. 

(\textrm{a}) The condition $\mathrm{Fix}(\varphi)=\{\xi\}$ implies that $\varphi$ is a parabolic isometry of $X$. In fact, if $\varphi$ is either elliptic or axial, we can find $x_0 \in X$ so that $d_\varphi(x_0)=|\varphi|$. Let $\gamma_0:\mathbb{R} \to X$ be the geodesic satisfying $\gamma_0(0)=x_0$ and $\gamma_0(\infty)=\xi$. Then $\gamma_0$ is preserved by $\varphi$ and therefore $\gamma_0(-\infty)$ is also a fixed point of $\varphi$, which contradicts $\mathrm{Fix}(\varphi)=\{\xi\}$. It follows from Lemma \ref{parabolic} that $\varphi$ leaves all horospheres centered at $\xi$ invariant.

(\textrm{b}) By Lemma \ref{parabolic} the group $\Gamma_0$ preserves all horospheres centered at $\xi$. This together with Lemma \ref{visible horo} yields that every unbounded sequence $\{\varphi_k(x)\}_{k \geq 1}, \varphi_k \in \Gamma_0$ converges to $\xi$ in the cone topology. Therefore the sequence $\{\gamma_k\}_{k \geq 1}$ of geodesic rays converges pointwisely to $\gamma$, where $\gamma_k$ is the unique geodesic ray emanating from $x$ that passes through $\varphi_k(x)$.

(\textrm{c}) For $t\geq0$, let $H_{-t}$ be the horosphere centered at $\xi$ such that $\gamma(t) \in H_{-t}$. In particular, $x \in H_0$. As discussed in (\textrm{a}), the orbit of $x$ is also contained in $H_0$. It is known that for any $t,s \in [0,\infty)$,
\[d(H_{-t},H_{-s})=|t-s|.\]

Thus, for every $\varphi \in \Gamma_0$ and for any $t\in [0,\infty)$, we have
\be \label{d-horo}
d(\varphi(x), \gamma(t)) \geq d(H_0,H_{-t})=t.
\ene

We are now ready to prove Proposition \ref{0-ent-1}.
\bp [Proof of Proposition \ref{0-ent-1}]
Without loss of generality we may assume that the set $\{\varphi(x),\varphi \in \Gamma_0\}$ is unbounded; otherwise the conclusion clearly follows. Firstly we claim that it suffices to show that
\be \label{A-ex-0}
\lim_{k \to \infty}\frac{\max_{|\varphi|_w=k} d(x, \varphi(x))}{k}=0.
\ene

If \eqref{A-ex-0} is true, then the proof is a simple observation as follows. Set
\[b_k=\max_{|\varphi|_w \leq k} d(x, \varphi(x)), \quad \forall \ k \geq 1.\]

Choose $\beta_k \in \Gamma_0$ with $|\beta_k|_w \leq k$ so that
\[b_k=d(x, \beta_k(x)).\]

Since we assume the unboundedness of $\{\varphi(x),\varphi \in \Gamma_0\}$, we have
\[\lim_{k \to \infty}|\beta_k|_w=\infty.\]
Together with \eqref{A-ex-0}, this yields
\[\lim_{k \to \infty}\frac{b_k}{|\beta_k|_w}=0.\]

Therefore
\[\lim_{k \to \infty}\frac{b_k}{k} \leq \lim_{k \to \infty}\frac{b_k}{|\beta_k|_w}=0.\]

The conclusion follows.

Now we start to prove \eqref{A-ex-0}. 

For every integer $k \geq 1$, we set
\begin{equation}
\label{def a_*} a_k=\max_{|\varphi|_w=k} d(x, \varphi(x)).
\end{equation}

We can choose $\varphi_k \in \Gamma_0$ with $|\varphi_k|_w=k$ so that
\[a_k=d(x, \varphi_k(x)).\] 

One may write $$\varphi_k=\Pi_{i=1}^k \varphi_{k;i}$$ where $\varphi_{k;i} \in S$ for all $1\leq i \leq k$. Since $|\varphi_k|_w=k$, we have for any $1\leq k_1\leq k_2\leq k$,
\begin{equation*}
|\Pi_{i=k_1}^{k_2} \varphi_{k;i}|_w=k_2-k_1+1.
\end{equation*}

Hence for all $k,l\geq 1$, by the triangle inequality
\begin{eqnarray*}
a_{k+l}&=&d(x,\varphi_{k+l}(x))\\
&=&d(x, \Pi_{i=1}^{k+l} \varphi_{k+l;i}(x)) \nonumber \\
&\leq& d(x, \Pi_{i=1}^{k} \varphi_{k+l;i}(x))+d(\Pi_{i=1}^{k} \varphi_{k+l;i}(x),\Pi_{i=1}^{k+l} \varphi_{k+l;i}(x)) \nonumber \\
&=& d(x, \Pi_{i=1}^{k} \varphi_{k+l;i}(x))+d(x,\Pi_{i=k+1}^{k+l} \varphi_{k+l;i}(x)) \nonumber \\
&\leq & a_k+a_l \nonumber
\end{eqnarray*}
where the last inequality follows by the definition \eqref{def a_*}. 

That is, the sequence $\{a_k\}_{k\geq 1}$ is subadditive. By Fekete's Subadditive Lemma, the limit 
\begin{equation}
\label{A-ex}
A=\lim_{k \to \infty}\frac{a_k}{k}
\end{equation} 
exists and $A \geq 0$.

We argue by contradiction that \[A=0.\]

Suppose that $$A>0.$$ By \eqref{A-ex} we know that for any $\epsilon \in (0,A)$, there exists an integer $K(\epsilon)>0$ such that 
\be \label{0-i-1}
(A-\epsilon)\cdot k \leq a_{k} \leq (A+\epsilon)\cdot k, \quad \forall \ k \geq K(\epsilon).
\ene
Moreover, the sequence $\{a_k-(A-\epsilon)\cdot k\}_{k \geq 1}$ is unbounded. So there is a subsequence  $\{a_{i_k}-(A-\epsilon)\cdot i_k\}_{k \geq 1}$ with $i_k \geq K(\epsilon)$  such that
\begin{equation*}
a_{i_k-j}-(A-\epsilon)\cdot (i_k-j) \leq a_{i_k}-(A-\epsilon)\cdot i_k, \quad \forall \ 0 \leq j < i_k,
\end{equation*}
which implies that
\[a_{i_k}-a_{i_k-j} \geq  (A-\epsilon)\cdot j.\]
Together with \eqref{0-i-1} and the subadditivity of $\{a_k\}_{k\geq 1}$, this yields
\begin{equation}
\label{uniform linear}
(A-\epsilon)\cdot j \leq a_{i_k}-a_{i_k-j} \leq (A+\epsilon)\cdot j, \quad \forall \ K(\epsilon) \leq j <i_k.
\end{equation}

Recall that $\varphi_{i_k}=\Pi_{i=1}^{i_k} \varphi_{i_k;i}$ is an element of $\Gamma_0$ with word norm $i_k$ satisfying $d(x,\varphi_{i_k}(x))=a_{i_k}$. For every $K(\epsilon) \leq j <i_k$ we consider the point $\Pi_{i=j+1}^{i_k} \varphi_{i_ki}(x)$. By the definition \eqref{def a_*} we have
\begin{equation}
\label{upper}
d(x, \Pi_{i=j+1}^{i_k} \varphi_{i_k;i}(x))\leq  a_{i_k-j}.
\end{equation}

On the other hand, by the triangle inequality
\begin{eqnarray}
\label{lower}
d(x, \Pi_{i=j+1}^{i_k} \varphi_{i_k;i}(x))&=&d(\Pi_{i=1}^{j}\varphi_{i_k;i}(x),\varphi_{i_k}(x)) \\
&\geq& d(x,\varphi_{i_k}(x))-d(x,\Pi_{i=1}^{j}\varphi_{i_k;i}(x))\nonumber\\
&\geq& a_{i_k}-a_{j}.\nonumber
\end{eqnarray}

Now set \[R_{k;j}:=a_{i_k}-d(x, \Pi_{i=j+1}^{i_k} \varphi_{i_k;i}(x)).\] Combining \eqref{uniform linear}, \eqref{upper} and \eqref{lower} we have
\begin{equation}
\label{R}
(A-\epsilon) \cdot j\leq a_{i_k}-a_{i_k-j}  \leq R_{k;j}\leq a_j \leq (A+\epsilon) \cdot j, \quad \forall \ K(\epsilon) \leq j < i_k.
\end{equation}
In particular, $R_{k;j}>0$.

For every $k \geq 1$, let $\gamma_{_k}:[0,\infty) \to X$ be the unique geodesic ray with $\gamma_{k}(0)=x$ and $\gamma_{k}(a_{i_k})=\varphi_{i_k}(x)$. For every $ K(\epsilon) \leq j < i_k$ we consider the point $\gamma_{k}(R_{k;j})$. The two points $\gamma_{k}(R_{k;j})$ and $\Pi_{i=1}^{j} \varphi_{i_k;i}(x)$ have the same distance to $\varphi_{i_k}(x)$. This is because that
\begin{eqnarray}
\label{equi d}
d(\gamma_{k}(R_{k;j}),\varphi_{i_k}(x))&=&a_{i_k}-R_{k;j} \\
&=& d(x,  \Pi_{i=j+1}^{i_k} \varphi_{i_k;i}(x)) \nonumber\\
&=& d(\Pi_{i=1}^{j}  \varphi_{i_k;i}(x), \varphi_{i_k}(x)).\nonumber 
\end{eqnarray}

Since $X$ is a CAT($0$) space, $\angle_{\gamma_{k}(R_{k;j})}(\Pi_{i=1}^{j}  \varphi_{i_k;i}(x),\varphi_{i_k}(x))$ is bounded from above by the corresponding angle of the Euclidean comparison triangle. By \eqref{equi d}, the Euclidean comparison triangle is an isosceles triangle. it follow that $\angle_{\gamma_{k}(R_{k;j})}(\Pi_{i=1}^{j}  \varphi_{i_k;i}(x),\varphi_{i_k}(x)) \leq \pi/2$ and hence
\begin{equation*}
\angle_{\gamma_{k}(R_{k;j})}(x, \Pi_{i=1}^{j}  \varphi_{i_k;i}(x)) \geq \dfrac{\pi}{2}.
\end{equation*}
See Figure \ref{A=0}.

\begin{figure}[!tbp]
\begin{tikzpicture}[xscale=0.6,yscale=1.2]

\draw [dashed] (0,3) arc [radius=30, start angle=90, end angle=70];
\draw [dashed] (0,3) arc [radius=30, start angle=90, end angle=110];

\draw [fill](-7.64,2)  circle [radius=0.04];
\draw [fill](2.2,2.92) circle [radius=0.04];
\draw [fill](7.64,2)   circle [radius=0.04];

\draw [thick] (-7.64,2) to [out=348, in=205] (2.2,2.92);
\draw [thick] (2.2,2.92) to [out=332, in=182] (7.64,2);
\draw [thick] (-7.64,2) to [out=344, in=196] (7.64,2);

\node [above] at (-7.85,2) {\scriptsize$x$};
\node [above] at (2.9,2.9) {\scriptsize$\Pi_{i=1}^{j}  \varphi_{i_k;i}(x)$};
\node [above] at (8.2,2){\scriptsize $\varphi_{i_k}(x)$};

\draw [fill](1.9,0.85) circle [radius=0.04];
\node [below] at (1.9,0.85) {\scriptsize$\gamma_{k}(R_{k;j})$};

\draw (2.2,2.92) to [out=310, in =40] (1.9,0.85);
\draw (1.6,0.82) to [out=60, in =180] (2.06,1);
\node [above] at (1.6,0.85){\scriptsize$\frac{\pi}{2}$};
\node [above] at (1.25,0.88){\tiny $\geq$};

\node [right] at (9.8,0.95){\scriptsize $H_0$};

\end{tikzpicture}
\caption{\label{A=0}}
\end{figure}

Consider the geodesic triangle with vertices $x, \gamma_{i_k}(R_{k;j})$ and  $\Pi_{i=1}^{j}  \varphi_{i_k;i}(x)$. Since $X$ is a CAT(0) space, we have for every $ K(\epsilon) \leq j <i_k$,
\begin{eqnarray}\label{0-i-4-4-4}
d(\gamma_{k}(R_{k;j}),\Pi_{i=1}^{j}  \varphi_{i_k;i}(x))&\leq& (d^2(x,\Pi_{i=1}^{j}  \varphi_{i_k;i}(x))-d^2(x,\gamma_{k}(R_{k;j})))^{1/2} \\
&\leq& (a_{j}^2-R_{k;j}^2)^{1/2} \nonumber \\
&\leq& ((A+\epsilon)^2\cdot j^2-(A-\epsilon)^2\cdot j^2)^{1/2} \nonumber \\
&=& 2 j \sqrt{A \epsilon}.\nonumber
\end{eqnarray}
Here we have used the estimates \eqref{0-i-1} and \eqref{R}.

Thus, by \eqref{R}, \eqref{0-i-4-4-4} and the triangle inequality we have for every $ K(\epsilon) \leq j < i_k$,
\bear \label{0-i-5}
d(\gamma_{k}(A\cdot j), \Pi_{i=1}^{j}  \varphi_{i_k;i}(x))&\leq & d(\gamma_{k}(A\cdot j), \gamma_{k}(R_{k;j}))+d(\gamma_{k}(R_{k;j}),\Pi_{i=1}^{j}  \varphi_{i_k;i}(x))\\
&\leq & |A\cdot j-R_{k;j}|+ 2 j \sqrt{A \epsilon}   \nonumber \\
&\leq & \epsilon \cdot j+2 j \sqrt{A \epsilon} \nonumber \\
&=& (\epsilon+2\sqrt{A\epsilon}) \cdot j. \nonumber
\eear

On the other hand, by Fact (\textrm{b}) we know that  the sequence $\{\gamma_k\}_{n\geq 1}$ of the geodesic rays converges pointwisely to the geodesic ray $\gamma:[0,\infty)\to X$ with $\gamma(0)=x$ and $\gamma(\infty)=\xi$. This implies that for every fixed integer $j \geq  K(\epsilon)$,
\be \label{0-i-6}
\lim_{k\to \infty} d(\gamma_k(A\cdot j), \gamma(A\cdot j))=0.
\ene

By \eqref{d-horo} and the triangle inequality we have  for every $  K(\epsilon) \leq j \leq i_k$ and for all $k\geq 1$,
\begin{eqnarray*}
A\cdot j&=& d(H_0,H_{-A \cdot j})\\ 
&\leq& d(\Pi_{i=1}^{j}  \varphi_{i_k;i}(x), \gamma(A \cdot j)) \nonumber \\
&\leq & d(\Pi_{i=1}^{j}  \varphi_{i_k;i}(x), \gamma_k( A\cdot j))+d(\gamma_k(A \cdot j), \gamma(A \cdot j)) \nonumber  \\
&\leq & (\epsilon+2\sqrt{A\epsilon}) \cdot j+d(\gamma_k(A \cdot j), \gamma(A \cdot j)),  \nonumber 
\end{eqnarray*}
where the last inequality follows from \eqref{0-i-5}. Taking the limit as $k \to \infty$, by \eqref{0-i-6} we get that
\[A\cdot j \leq (\epsilon+2\sqrt{A\epsilon}) \cdot j.\]
That is,
\[A \leq \epsilon+2\sqrt{A\cdot \epsilon}\]
Observe that this inequality is valid for any $\epsilon \in (0,A)$. By taking $\epsilon \to 0$ we obtain
$$A=0,$$which contradicts the assumption that $$A>0.$$

Therefore $A=0$.
The proof of Proposition \ref{0-ent-1} is complete.
\ep

The second part of the proof of Proposition \ref{sec-0-ent}  is a standard volume comparison argument.

\bp [Proof of Proposition \ref{sec-0-ent}]
Fix a reference point $x \in \widetilde{M}$. Set 
\[\epsilon_0=\frac{ \inf_{\varphi \in \Gamma_0 \setminus \{e\}} d(x, \varphi(x))}{4}.\]
Since $\Gamma_0$ acts freely and properly discontinuously on $\widetilde{M}$, we have that 
\begin{equation*}
\epsilon_0>0.
\end{equation*}

Let $S$ be a finite symmetric generating set for $\Gamma_0$ and denote by $|\cdot|_w$ the word norm with respect to $S$. Recall that  $B_w(k)=\{\varphi \in \Gamma_0: \ |\varphi|_w \leq k\}$ is the geodesic ball in the word metric. 

For every $k \geq 1$, set
\[b_k=\max_{\varphi \in B_w(k)} d(x, \varphi(x)).\]

By the triangle inequality, the union of geodesic balls satisfies that
\be \label{0-0-i-2}
\bigcup_{\varphi \in B_w(k)} B_{\varphi(x)}(\epsilon_0) \subset B_x(b_k+\epsilon_0).
\ene

Recall that $n$ is the dimension of $\widetilde{M}$. Since the curvature of $\widetilde{M}$ is nonpositive, by the volume comparison theorem there exists a constant $C>0$ depending only on $n$ and $\epsilon_0$, such that for all $\varphi \in \Gamma_0$,
\be \label{0-1-i-2}
\Vol (B_{\varphi(x)}(\epsilon_0)) \geq C.
\ene

On the other hand, the Ricci curvature satisfies $\Ric_{\widetilde{M}} \geq -(n-1)$, it follows by the standard Bishop-Gromov volume comparison inequality \cite{Gromov-book} that there exists a constant $C'>0$ depending only on $n$ such that
\bear \label{0-1-i-3}
\Vol (B_x(b_k+\epsilon_0)) &\leq & \Vol_{\H^{n}}(B(b_k+\epsilon_0))  \\
&\leq & C' \exp ((n-1)\cdot (b_k+\epsilon_0)), \nonumber
\eear
where $B(b_k+\epsilon_0)$ is a geodesic ball of radius $b_k+\epsilon_0$ in the $n$-dimensional hyperbolic space $\H^n$.

By the choice of $\epsilon_0$ we know that 
\be \label{0-0-i-3}
B_{\varphi_1(x)}(\epsilon_0) \cap B_{\varphi_2(x)}(\epsilon_0)=\emptyset, \quad \forall \ \varphi_1 \neq \varphi_2.
\ene
Otherwise, there would exist a point $y \in B_{\varphi_1(x)}(\epsilon_0) \cap B_{\varphi_2(x)}(\epsilon_0)$ and therefore by the definition of $\epsilon_0$ we have
\beqar
4\epsilon_0 &\leq& d(x, \varphi_1^{-1} \varphi_2(x)) \\
&=& d(\varphi_1(x),\varphi_2(x))\\
&\leq& d(y, \varphi_1(x))+d(y, \varphi_2(x)) \\
&\leq& 2\epsilon_0,
\eeqar
which is impossible.

Combining \eqref{0-0-i-2}, \eqref{0-1-i-2}, \eqref{0-1-i-3} and \eqref{0-0-i-3} we obtain that
\bear \label{0-0-i-4}
 C' \exp ((n-1)\cdot (b_k+\epsilon_0)) &\geq & \Vol (B_x(b_k+\epsilon_0))\\
&\geq& \Vol(\bigcup_{\varphi \in B_w(k)} B_{\varphi(x)}(\epsilon_0)) \nonumber \\
&=&\sum_{\varphi \in B_w(k)} \Vol (B_{\varphi(x)}(\epsilon_0) \nonumber \\
&\geq & C \cdot \#B_w(k). \nonumber
\eear
Taking the logarithm on \eqref{0-0-i-4} we get
\begin{equation}
\label{log compare}
\dfrac{\log(\#B_x(k))}{k} \leq \dfrac{\log(C'/C)+(n-1) \cdot (b_k+\epsilon_0)}{k}.
\end{equation}
By Proposition \ref{0-ent-1} we have that $$\lim_{k \to \infty} \dfrac{b_k}{k}=0.$$ Now taking the limit of the inequality \eqref{log compare} yields 
\bear \label{0-0-i-5}
\limsup_{k \to \infty}\frac{\log (\#B_w(k))}{k}= 0, \nonumber
\eear
that is,
$$h(\Gamma_0)=0.$$

The proof is complete.
\ep

Let $M$ be the manifold in Conjecture \ref{Eb-C}. If $\dim(M)=4$, by \cite[Theorem 3.1 or Lemma 3.1g]{Eberlein80} it is known that each end of $M$ is homotopic to a closed $3$-dimensional aspherical manifold. Actually Cerbo in \cite{Cerb09} showed that the fundamental group of a closed $3$-dimensional aspherical manifold either is almost nilpotent or has a uniform positive algebraic entropy. Thus for the 4-dimensional case, the first part of Theorem \ref{mt-1} follows from Proposition \ref{sec-0-ent} and the work of Di Cerbo in \cite{Cerb09}. We remark that the results in \cite{Cerb09} rely on Perelman's solution of the Poincare Conjecture. For general dimensions, in the following section we will complete the proof of Theorem \ref{mt-1} without using Perelman's solution of the Poincare Conjecture. In other words, we provide two different proofs to address Conjecture \ref{mt-1} for the case that $\dim(M)=4$.


\section{Proof of Theorem \ref{mt-1}}\label{sec-mt}
In this section we prove Theorem \ref{mt-1}.

Throughout this section we assume that $\widetilde{M}$ is a complete, simply connected visibility manifold with $-1 \leq K_{\widetilde{M}}\leq 0$ and that $\Gamma$ is a \emph{nonuniform lattice} in $\widetilde{M}$, i.e., $\Gamma$ acts freely and properly discontinuously by isometries on $\widetilde{M}$ and $\widetilde{M}/\Gamma$ is a noncompact manifold with finite volume. Nonuniform lattices in visibility manifolds are investigated in \cite{Eberlein80}. It was shown that $\widetilde{M}/\Gamma$ has only finitely many ends and each end is parabolic and Riemannian collared. 

More precisely, let $E$ be an end of $\widetilde{M}/\Gamma$ and $\sigma:[0,\infty) \to \widetilde{M}/\Gamma$ be a geodesic ray that converges to $E$. Choose $\gamma$ a lift of $\sigma$ to $\widetilde{M}$ and let $\eta=\gamma(\infty) \in \widetilde{M}(\infty)$. The fundamental group $\pi_1(E)$ of $E$ is isomorphic to the stability subgroup $\Gamma_\eta=\{\varphi \in \Gamma: \varphi(\eta)=\eta\}$, and the end $E$ can be regarded as $B_\eta/\Gamma_\eta$ for a horoball $B_\eta$ centered at $\eta$.
$\Gamma_\eta$ contains only parabolic elements. Therefore by Lemma \ref{parabolic} all horoballs centered at $\eta$ are invariant under $\Gamma_\eta$. Moreover, $\Gamma_\eta$ acts cocompactly on every horosphere centered at $\eta$. 

It is implied by Proposition \ref{C Margulis} that $\Gamma_\eta$ is almost nilpotent if $d_\varphi(\gamma)$ is uniformly bounded for all $\varphi \in \Gamma_\eta$ along some geodesic ray $\gamma \in \eta$. For the general case, the action of $\Gamma_\eta$ on a horosphere $H$ does not seem to give us much useful information on the structure of $\Gamma_\eta$. The main reason is that $H$ with the induced metric might have curvature of both signs and even conjugate points (e.g., see \cite{HH-77}). In order to overcome this difficulty we consider the transverse space $\widetilde{M}_\eta$ of $\eta$. By Theorem \ref{Transverse proper}, $\widetilde{M}_\eta$ is a complete CAT($0$) space of bounded geometry. As remarked in Section \ref{sec-pre}, if a group of isometries  fixing $\eta$ acts cocompactly on a horosphere centered at $\eta$, then it also acts cocompactly on $\widetilde{M}_\eta$. Thus, the group $\Gamma_\eta$ induces a group of isometries acting cocompactly on $\widetilde{M}_\eta$.

In contrast with the case of manifolds, there are many proper CAT($0$) spaces with infinite-dimensional boundaries. Thus it is difficult to relate an isometry group with its action on the geometric boundary. 

A \emph{flat} in a CAT($0$) space is a complete, totally geodesic, isometrically embedded Euclidean space. A finitely generated group $\Gamma$ is called \emph{amenable} if $\Gamma$ has a finite additive left-invariant probability measure. The following theorem due to Adams and Ballmann is essential in the proof of Theorem \ref{mt-1}.

\begin{theorem} \cite{AB98}
\label{Adams Ballmann}
Let $X$ be an unbounded proper CAT($0$) space. If $\Gamma$ is an amenable group of isometries of $X$, then $\Gamma$ fixes either a point in its geometric boundary $X(\infty)$ or a flat in $X$ setwisely.
\end{theorem}

Before proving Theorem \ref{mt-1}, we first provide the following result which can be viewed as a type of classical Bieberbach Theorem. 
\begin{proposition} \label{BT-np}
Let $\Gamma$ be a finitely generated group which acts on the Euclidean space $\R^n$ by isometries. If there exists a bounded subset $F \subset \R^n$ such that $\cup_{\varphi \in \Gamma}\varphi(F)=\R^n$. Then we have for any $o \in \R^n$,
\begin{equation}\label{BT->0}
\limsup_{k \to \infty}\max_{|\varphi|_w=k}\frac{d(o,\varphi(o))}{k}>0.
\end{equation}
\end{proposition} 

\bp [Proof of Proposition \ref{BT-np}]
By the triangle inequality, for any distinct points  $p, q \in \R^n$,
\beqar
\lim_{k \to \infty}|\max_{|\varphi|_w=k}\frac{d(p,\varphi(p))}{k}-\max_{|\varphi|_w=k}\frac{d(q,\varphi(q))}{k}|&\leq& \lim_{k \to \infty}\max_{|\varphi|_w=k}|\frac{d(p,\varphi(p))}{k}-\frac{d(q,\varphi(q))}{k}|\\
&\leq & \lim_{k \to \infty} \frac{2d(p,q)}{k}=0.
\eeqar
So it suffices to show \eqref{BT->0} when $o$ is the origin of $\mathbb{R}^n$.

By the boundedness of $F$ one may choose $R>0$ sufficiently large such that
\bear
\label{ball cover}
\bigcup_{\varphi \in \Gamma} \varphi(B_o(R))=\bigcup_{\varphi \in \Gamma}  B_{\varphi(o)}(R)=\R^n
\eear
where $B_o(R)=\{x \in \mathbb{R}^n: \ d(o,x)\leq R\}$. 

Let $S_o(2R)=\{x: \ d(o,x)= 2R\}$ be the sphere of radius $2R$ centered at $o$. We can find a finite collection of points $\{p_i\}_{1\leq i \leq l}\subset S_o(2R)$, where $l>0$ is an integer, such that
\bear \label{l-1-0-0}
\bigcup_{1\leq i \leq l}\cone_o(p_i, \frac{\pi}{6})=\R^n
\eear
where $\cone_o(x,\theta)=\{y \in \R^n: \ \angle_o(x,y)\leq \theta\}$ is the closed cone about $x$ of angle $\theta$.

By \eqref{ball cover}, for every $1 \leq i \leq l$, there exists an element $\varphi_i \in \Gamma$ such that
\[\varphi_i(o) \in B_{p_i}(R).\]

It then follows from the triangle inequality that
\begin{equation*}
R\leq d(o,\varphi_i(o)) \leq 3R.
\end{equation*}

From the cosine formula it is easy to see that for every $1\leq i \leq l$,
\begin{equation*}
\angle_o(p_i, \varphi_i(o)) \leq \frac{\pi}{6}.
\end{equation*}
This together with \eqref{l-1-0-0} implies that
\bear \label{l-1-0-1}
\bigcup_{1\leq i \leq l}\cone_o(\varphi_i(o), \frac{\pi}{3})=\R^n.
\eear

Let $L\geq 1$ be the maximal word norm of $\{\varphi_i\}_{1\leq i \leq l}$. That is,
\[L=\max_{1\leq i \leq l} |\varphi_i|_w>0.\]

Now we construct a sequence $\{\Phi_k \}_{k \geq 1} \subset \Gamma$ satisfying
\begin{equation*}
\inf_{k \geq 1} \frac{d(o,\Phi_k(o))}{|\Phi_k|_w}>0 \text{ and } \lim_{k \to \infty} |\Phi_k|_w=\infty.
\end{equation*}

Let $\Phi_1=\varphi_{j_1}$ with $j_1=1$. Consider the straight (geodesic) segment from $\Phi_1^{-1}(o)$ to $o$ whose length is  $d(\varphi_1^{-1}(o), o)\geq R$. By \eqref{l-1-0-1} there exists $1\leq j_2 \leq l$ such that 
\begin{eqnarray*}
\angle_o(\varphi_{j_2}(o),\Phi_1^{-1}(o)) \geq \frac{2 \pi}{3}.
\end{eqnarray*}
A direct computation shows that 
\bear \label{l-1-0-3}
d(\varphi_{j_2}(o),\Phi_1^{-1}(o)) &\geq& d(\Phi_1^{-1}(o),o)+\frac{d(\varphi_{j_2}(o),o)}{2} \\
&\geq &\frac{3R}{2}.\nonumber
\eear

Let $\Phi_2=\Phi_1  \varphi_{j_2}$. By the definition of $L$ and the triangle inequality it is clear that 
\beqar
|\Phi_2|_w &\leq& L+1 \\
 &\leq& 2L.
\eeqar

Thus, we have 
\begin{equation*}
d(\Phi_2(o),o) \geq \frac{3R}{2} \quad \text{and} \quad |\Phi_2|_w \leq 2L.
\end{equation*}

Assume now that we already have $\Phi_k \in \Gamma$ with $d(\Phi_k(o),o) \geq (k+1)R/2$ and $|\Phi_k|_w \leq kL$. By \eqref{l-1-0-1} we may choose $\varphi_{j_{k+1}} \in \{\varphi_1,\cdots,\varphi_l \}$ so that $\angle_o(\varphi_{j_{k+1}}(o),\Phi_k^{-1}(o)) \geq \frac{2 \pi}{3}$. By a computation similar to that of \eqref{l-1-0-3} we have
\bear \label{l-1-0-4}
d(\varphi_{j_{k+1}}(o),\Phi_k^{-1}(o)) &\geq& d(\Phi_k^{-1}(o),o)+\frac{d(\varphi_{j_{k+1}}(o),o)}{2} \\
&\geq &\frac{(k+2)R}{2}.\nonumber
\eear

Now set $\Phi_{k+1}=\Phi_k \cdot \varphi_{j_{k+1}}$. It is clear that $|\Phi_{k+1}|_w \leq (k+1)L$. By repeating this process we get a sequence $\{\Phi_k \}_{k \geq 1} \subset \Gamma$ such that
\begin{equation*}
d(o,\Phi_k(o)) \geq \dfrac{(k+1)R}{2} \text{ and } |\Phi_k|_w \leq kL.
\end{equation*}
And by \eqref{l-1-0-4} we know that 
\beqar
\lim_{k \to \infty}|\Phi_{k}|_w=\infty.
\eeqar 
Hence
\begin{eqnarray*}
\limsup_{k \to \infty} \max_{|\varphi|_w=k} \dfrac{d(o,\varphi(o))}{k} &\geq& \limsup_{k \to \infty} \dfrac{d(o,\Phi_k(o))}{|\Phi_k|_w} \\
&\geq& \dfrac{R}{2L}>0.\nonumber
\end{eqnarray*}
This completes the proof.
\ep

\begin{remark}\label{BF21-f}
It was shown in \cite[Proposition 1.10]{MR4275871} that for a finite set $S$ of isometries of a Euclidean space, either $S$ has a common fixed point, or the asymptotic joint displacement of $S$ vanishes. This in particular implies Proposition \ref{BT-np}. We are very grateful to one anonymous referee for bringing this stronger result of Breuillard-Fujiwara to our attention.
\end{remark}

\begin{remark}
Unlike the assumptions in the classical Bieberbach Theorem \cite[Page 103]{BGS}, in Proposition \ref{BT-np} the action of $\Gamma$ on $\R^k$ is not required to be properly discontinuous. 
\end{remark}

We are now ready to prove Theorem \ref{mt-1}.

\begin{proof}[Proof of Theorem \ref{mt-1}]
Let $\widetilde{M}$ be the universal cover of $M$.  For an end $E$ of $M$ we let $\sigma:[0,\infty) \to \widetilde{M}/\Gamma$ be a geodesic ray that converges to $E$. Choose a geodesic ray $\gamma \subset \widetilde{M}$ which is a lift of $\sigma$ and let $\eta=\gamma(\infty) \in \widetilde{M}(\infty)$. As introduced above, by \cite[Theorem 3.1]{Eberlein80}, the fundamental group $\pi_1(E)$, consisting of parabolic isometries, is isometric to  $\Gamma_\eta=\{\varphi \in \Gamma:\ \varphi(\eta)=\eta\}$. Meanwhile, it is known \cite[Lemma 3.1g]{Eberlein80} that $\Gamma_\eta$ acts properly and cocompactly on every horosphere centered at $\eta$. In particular, the group $\Gamma_\eta$ acts cocompactly on the transverse space $\widetilde{M}_\eta$ of $\eta$.

Our aim is to show that the transverse space $\widetilde{M}_\eta$ of $\eta$ is a bounded space, which in particular implies that 
\begin{equation*}
\sup_{\varphi \in \Gamma_\eta} d_\varphi(\gamma) \leq C \text{ for some } C \geq 0.
\end{equation*}
Theorem \ref{mt-1} then follows immediately from Proposition \ref{C Margulis}.

To this end, we consider the action of $\Gamma_\eta$ on $\widetilde{M}_\eta$. As remarked in Section 2, an isometry $\varphi \in \Gamma_\eta$ of $\widetilde{M}$ induces an isometry of $\widetilde{M}_\eta$ with the same translation length. Recall that Theorem \ref{w-tl} tells that the translation length $|\varphi|=0$ on $\widetilde{M}$ for every $\varphi \in \Gamma_\eta$. Thus, the translation length of the induced isometry $\varphi$ of $\widetilde{M}_\eta$ also vanishes. This implies that $\Gamma_\eta$ contains only non-axial isometries of $\widetilde{M}_\eta$.  

We argue by contradiction to show that  \emph{the transverse space $\widetilde{M}_\eta$ is bounded}.

Assume that $\widetilde{M}_\eta$ is unbounded. 

First by Theorem \ref{0-ent-v} we know that the group $\Gamma_\eta$ has zero algebraic entropy. In particular, it is an amenable group \cite[Page 205]{Harpe-book}. From Theorem \ref{Transverse proper} we know that the transverse space $\widetilde{M}_\eta$ is a complete proper CAT(0) space. By Theorem \ref{Adams Ballmann}, the group $\Gamma_\eta$ either fixes a point in the geometric boundary $\widetilde{M}_\eta(\infty)$, or fixes setwisely a flat in $\widetilde{M}_\eta$.\\

Case-1. There exists a point $\eta_1 \in \widetilde{M}_\eta(\infty)$ such that $\Gamma_\eta$ fixes $\eta_1$.

Since the translation length of every element of $\Gamma_\eta$ is $0$ on $\widetilde{M}_\eta$, an isometry $\varphi \in \Gamma_\eta$ which acts nontrivially on $\widetilde{M}_\eta$ is either an elliptic or a parabolic isometry. By Lemma \ref{parabolic}, each horosphere in $\widetilde{M}_\eta$ centered at $\eta_1$ is invariant under $\Gamma_\eta$. This contradicts that $\Gamma_\eta$ acts cocompactly on $\widetilde{M}_\eta$ (see Remark \ref{re-nocpt}).\\

Case-2. The group $\Gamma_\eta$  fixes setwisely a flat $\R^m \subset \widetilde{M}_\eta$.

Since  $\Gamma_\eta$ acts cocompactly on $\widetilde{M}_\eta$, it also acts cocompactly on $\R^m$. By Proposition \ref{BT-np} there is a point $o \in \R^m \subset \widetilde{M}_\eta$ such that 
\begin{equation}\label{s-6->0}
\limsup_{k \to \infty}\max_{\varphi \in \Gamma_\eta, \ |\varphi|_w=k}\frac{d_\eta(o,\varphi (o))}{k}>0,
\end{equation}
Where $d_\eta$ is the distance function on $\widetilde{M}_\eta$.

On the other hand, fix a horosphere $H$ in $\widetilde{M}$ centered at $\eta$. Since $\widetilde{M}_\eta$ is the metric completion of $(H,d_\eta)$, we can find a point $p \in H$ so that $d_\eta(p,o) \leq 1$. Combining Proposition \ref{0-ent-1} and \eqref{s-6->0} yields
\beqar
0 &=& \lim_{k \to \infty}\max_{|\varphi|_w\leq k}\frac{d(p,\varphi(p))}{k}  \\
&\geq& \limsup_{k \to \infty}\max_{|\varphi|_w\leq k}\frac{d_\eta(p,\varphi(p))}{k}  \\
&\geq& \limsup_{k \to \infty} \max_{|\varphi|_w=k} \frac{d_\eta(o,\varphi(o))-d_\eta(p,o)-d_\eta(\varphi(p),\varphi(o))}{k}  \\
&\geq& \limsup_{k \to \infty} \max_{|\varphi|_w=k} \frac{d_\eta(o,\varphi(o))-2}{k}  \\
&= & \limsup_{k \to \infty}\max_{|\varphi|_w=k}\frac{d_\eta(o,\varphi(o))}{k}  \\
&>&0. \eeqar
Which is a contradiction.\\

The proof is complete.
\end{proof}

\begin{remark}
Recall that $\widetilde{M}_\eta$ is a complete CAT(0) space. So from the proof of Theorem \ref{mt-1} we know that $\widetilde{M}_\eta$ is a bounded $\Gamma_\eta$-invariant CAT(0) space. Then it follows from the classical Cartan Fixed Point Theorem  (e.g., see \cite[Page 179]{BH-book})  that $\Gamma_\eta$ has a common fixed point in $\widetilde{M}_\eta$.  
\end{remark}

\section{Proofs of Corollary \ref{c-rank}, \ref{c-nov} and \ref{c-zero}}\label{pofcos} 
In this section we prove Corollary \ref{c-rank}, \ref{c-nov} and \ref{c-zero}.

First we prove Corollary \ref{c-rank}. We modify the argument in \cite{Sch84} for the pinched negative curvature case.
\begin{proof}[Proof of Corollary \ref{c-rank}]
By Theorem \ref{mt-1} one may assume that $\Gamma'$ be a nilpotent subgroup of $\pi_1(E)$ of finite index. So it suffices to show that the rank $\rank(\Gamma')$ of $\Gamma'$ is $\dim(M)-1$. By \cite[Theorem 3.1]{Eberlein80} we know that $\pi_1(E)$ is finitely presented. Thus $\Gamma'$ is a finitely generated torsion-free nilpotent group. By a theorem of Malcev \cite[Theorem II.2.18]{Rag72}) we have that $\Gamma'$ is isomorphic to a lattice of a simply connected nilpotent Lie group whose dimension is the same as the rank of $\Gamma'$. By \cite[Lemma 3.1g]{Eberlein80} we know that $\Gamma'$ acts on a horosphere of the universal cover of $M$, which is homeomorphic to $\R^{\dim(M)-1}$, with a compact manifold quotient. The conclusion that $\rank(\Gamma')=\dim(M)-1$ follows from the fact that any simply connected nilpotent Lie group of rank $d$ is homeomorphic to $\mathbb{R}^{d}$.
\end{proof}

In \cite{AS} Abresch and Schroeder constructed certain four dimensional open manifold $M^4$ of finite volume with curvature $-1\leq K_{M^4}<0$. And the fundamental group of each end of $M^4$ contains a subgroup which is isomorphic to the fundamental group of closed surface of genus $g\geq 2$. One may see \cite{AS} for more details. Now we prove Corollary \ref{c-nov}.
\bp [Proof of Corollary \ref{c-nov}]
First by \cite[Theorem 0.1]{AS} there exists a compact subset $C$ of $M^4$ such that $M^4\setminus C$ is homeomorphic to $W\times (0, \infty)$ where $W$ is a closed three dimensional aspherical manifold. Using this parameter it is known that as $t \to \infty$, the set $W\times (t, \infty)$ leaves every compact subset of $M^4$. 

We argue by contradiction.

Suppose that the universal covering space $\widetilde{M}^4$ of $M^4$ is a visibility manifold. Then we can apply \cite[Theorem 3.1]{Eberlein80} to get that there exists a boundary point $\eta \in \widetilde{M}^4(\infty)$ and precisely invariant horosphere $H$, determined by a geodesic ray whose infinite endpoint is $\eta$, such that an end $E$ is homeomorphic to $(H/\pi_1(E))\times(0,\infty)$ where $\pi_1(E)$ is the fundamental group of $E$. Using this parameter it is known that as $s \to \infty$, the set $(H/\pi_1(E))\times (s, \infty)$ leaves every compact subset of $M^4$. 

Fix the subset $W\times (1, \infty)$. Then there exists a large enough constant $s_0>0$ such that the following inclusion
\[(H/\pi_1(E))\times (s_0, \infty) \hookrightarrow W\times (1, \infty)\]
holds.

Similarly there exists a large enough constant $t_0>0$ such that we have the inclusion
\[W\times (t_0, \infty) \hookrightarrow (H/\pi_1(E))\times (s_0, \infty).\]

It is clear that the composition of inclusions
\[W\times (t_0, \infty) \hookrightarrow (H/\pi_1(E))\times (s_0, \infty) \hookrightarrow W\times (1, \infty)\]
is a homotopy equivalence. In particular, we have
\[\pi_1(W)\cong \pi_1 (H/\pi_1(E)) \cong \pi_1(E).\]

By the construction in \cite{AS} we know that $\pi_1(W)$ contains a free subgroup of rank at least two. In particular, the group $\pi_1(E)$ cannot be almost nilpotent \cite{Gromov-pg}. On the other hand, by Theorem \ref{mt-1} we know that $\pi_1(E)$ is almost nilpotent, which is a contradiction.
\ep

\bp [Proof of Corollary \ref{c-zero}]
Let $E$ be the end and $\sigma:[0,\infty)\to M$ be a geodesic ray converging to $E$. Then for any two lifts $\gamma_1,\gamma_2$ of $\sigma$ in the universal cover $\widetilde{M}$ of $M$ satisfying $\gamma_1(\infty)=\gamma_2(\infty)=\eta \in \widetilde{M}(\infty)$, by \cite[Theorem 3.1]{Eberlein80} there exists an element $\varphi \in \pi_1(E)$ such that $\varphi \circ \gamma_1=\gamma_2$. By Theorem \ref{mt-1} the group $\pi_1(E)$ is almost nilpotent. By \cite[Lemma 3.1g]{Eberlein80} the fundamental group $\pi_1(E)$ acts cocompactly on every horosphere centered at $\eta$. Then the conclusion follows by Theorem \ref{nil-to-0}.
\ep


\section{Appendix: a visibility manifold with a finite volume quotient is not Gromov hyperbolic}

In this appendix we construct a complete surface of finite volume with curvature $-1 \leq K <0$, whose universal cover is a visibility manifold but not a Gromov hyperbolic space. This example shows that the argument in the proof of  Corollary \ref{Gromov hyperbolic} is not adapted to general visibility manifolds.

Let $(S,g)$ be a noncompact surface of constant negative curvature $-1$ with finite volume. Such a surface has only finitely many cusps. For simplicity we further assume that $S$ has only one cusp $E$, which can be expressed as $\mathbb{S}^{1} \times [0,\infty)$ endowed with the hyperbolic metric $\exp(-2t) \mathrm{d}\tilde{s}^2+\mathrm{d}t^2$, where $\mathrm{d}\tilde{s}^2$ is the flat metric on $\mathbb{S}^{1}$. 

Let $h:[0,\infty) \to \mathbb{R}$ be a smooth function such that
\begin{IEEEeqnarray*}{RRL}
(\textrm{1})& & \;\; h \textrm{ is positive and monotonically decreasing},\\
(\textrm{2})& & \;\; h''/h \textrm{ is positive and monotonically decreasing},\\
(\textrm{3})& & \;\; h=\exp(-t) \textrm{ for } 0 \leq t \leq 1, \\
(\textrm{4})& & \;\; h=\dfrac{1}{t^{2}} \textrm{ for } t \geq 3.
\end{IEEEeqnarray*}
Such a function could be constructed by elementary calculus.

We change the metric on $E=\mathbb{S}^{1} \times [0,\infty)$ to $h^2(t) \mathrm{d}\tilde{s}^2 + \mathrm{d}t^2$ and obtain a new smooth metric $g'$ on $S$. Since both $\mathbb{S}^{1}$ and $[0,\infty)$ are complete, their warped product is also complete (e.g., see \cite[Lemma 7.2]{BO-wp}). We first show that $(S,g')$ is a finite volume surface with bounded nonpositive curvature whose universal cover is a visibility manifold.
 
Let $S'=S - (\mathbb{S}^{1} \times (3,\infty))$. $S'$ is a compact surface with boundary. The volume of $S$ with respect with the new metric $g'$ is
\begin{eqnarray*}
\textrm{Vol}_{g'}(S)&=& \textrm{Vol}_{g'}(S')+\textrm{Vol}_{g'}(\mathbb{S}^{1} \times (3,\infty)) \\
&=& \textrm{Vol}_{g'}(S') + \int_0^1 \int_3^\infty h(t) \,\mathrm{d}s \mathrm{d}t \\
&=& \textrm{Vol}_{g'}(S') + \int_3^\infty \dfrac{1}{t^2} \,\mathrm{d}t \\
&=& \textrm{Vol}_{g'}(S') + \dfrac{1}{3}<\infty.
\end{eqnarray*}

To see that $(S,g')$ is visible, we check that it satisfies a visibility criterion due to Eberlein and O'Neill \cite[Proposition 5.9]{EO-v}, which states that if a nonpositively curved manifold $M$ has curvature order $\leq 2$ at a point $p \in M$, i.e., if $$\int_1^\infty |k(\gamma_\omega(t))| t \,\mathrm{d}t = \infty \textrm{ for all } \omega \in \mathrm{S}_p M,$$
where $\gamma_\omega$ is the geodesic ray with initial velocity $\omega$ and $k(\gamma_\omega(t))=\min\{|K(\pi)|: \pi \subset \mathrm{T}_{\gamma_\omega(t)} \text{ is a two-dimensional subspace}\}$. Then $M$ is a visibility manifold.

In fact, direct computation gives 
\begin{equation*}
\label{K}
K_{g'}(q)=
\begin{cases}
\;\;-1 & \text{for } q \in S \setminus E\\
-\dfrac{h''(t)}{h(t)} & \text{for } q \in \mathbb{S}^{1} \times \{t\} \subset E
\end{cases}
\end{equation*}
where $K_{g'}(q)$ is the Gaussian curvature at $q$ with respect with the metric $g'$. By the choice of $h$, $-1 \leq K_{g'}<0$. We have for any $p \in   \mathbb{S}^{1} \times \{0\}$ and for any $\omega \in \mathrm{S}_p S$,
\begin{eqnarray*}
\int_1^\infty |K_{g'}(\gamma_\omega(t))| t \,\mathrm{d}t &\geq& \int_1^\infty \dfrac{h''}{h}(t) t \,\mathrm{d}t \\
&\geq& \int_3^\infty (\dfrac{1}{t^2})''/(\dfrac{1}{t^2}) \cdot t \, \mathrm{d}t\\
&=& \int_3^\infty \dfrac{6}{t} \,\mathrm{d}t = \infty,
\end{eqnarray*}
hence by the visibility criterion the universal cover of $(S,g')$ is a visibility surface.

It remains to show that $(S,g')$ is not Gromov hyperbolic. Let $\tilde{S}$ be the universal cover of $(S,g')$. Assume now that $\tilde{S}$ is $\delta$-hyperbolic for some $\delta>0$. Let $\gamma$ be a geodesic ray in $S$ converging to the end $E$ and $\tilde{\gamma}$ be a lift of $\gamma$ to $\tilde{S}$. It is not hard to see that a horoball in $\tilde{S}$ centered at $\tilde{\gamma}(\infty)$ is given by $\mathbb{R} \times [0,\infty)$ endowed with the metric $h^2(t)\mathrm{d}s^2+\mathrm{d}t^2$. Denote by $H_{-t}$ the horosphere $\mathbb{R} \times \{t\}$ and by $d_H$ the horospherical distance between two points on the same horosphere.

Let $T>3+12 \delta$ be a sufficiently large constant. Set $x_0=(0,T)$ and $y_0=(2 \delta T^2,T)$. We have $$d_H(x_0,y_0)=\int_0^{2 \delta T^2} \dfrac{1}{T^2} \,\mathrm{d}s= 2\delta.$$ 
Where $d_H(\cdot, \cdot)$ is the induced distance on the horosphere $H$.

Let $x=(0,T-2\delta)$ and $y=(2 \delta T^2,T-2\delta)$. The projection of the geodesic segment $\gamma_{x,y}$ onto the horosphere $H_{-T}$ is exactly the horospherical geodesic segment from $x_0$ to $y_0$. Together with a standard infinitesimal argument, this yields that 
\begin{equation}
\label{spherical actual}
d(x,y) > d_H(x_0,y_0)=2\delta. 
\end{equation}

On the other hand, \begin{equation*}
\label{upper 6}
d(x,y) \leq d_H(x,y) =\dfrac{2\delta T^2}{(T-2\delta)^2}<3 \delta.
\end{equation*}

Set $w=(0,T-8\delta)$ and $z=(2 \delta T^2,T-8\delta)$, we have 
\begin{equation*}
d(x,w)=d(y,z)=6 \delta \geq 2d(x,y)
\end{equation*}
and 
\begin{equation*}
d(z,w) \leq d_H(z,w)=\int_{0}^{2 \delta T^2} \dfrac{1}{(T-8\delta)^2} \, \mathrm{d}s= 2\delta \cdot \dfrac{T^2}{(T-8\delta)^2}.
\end{equation*}
Together with \eqref{spherical actual} we obtain
\begin{eqnarray*}
d(z,w)-d(x,y) &\leq& 2\delta \cdot \dfrac{T^2}{(T-8\delta)^2}-2\delta \\
&=& 32 \delta^2 \cdot \dfrac{T- 4\delta}{(T-8\delta)^2},
\end{eqnarray*}
which tends to $0$ as $T \to \infty$. This contradicts Lemma \ref{Gromov comparison}. Therefore $\tilde{S}$ cannot be Gromov hyperbolic. See Figure \ref{ABCD}.

\begin{figure}[!tbp]
\begin{tikzpicture}[xscale=0.6,yscale=1.2]
\draw [dashed] (0,0) arc [radius=15, start angle=90, end angle=80];
\draw [dashed] (0,0) arc [radius=15, start angle=90, end angle=100];
\draw [dashed] (0,1) arc [radius=16, start angle=90, end angle=76];
\draw [dashed] (0,1) arc [radius=16, start angle=90, end angle=104];
\draw [dashed] (0,3) arc [radius=18, start angle=90, end angle=72];
\draw [dashed] (0,3) arc [radius=18, start angle=90, end angle=108];

\draw [dashed, thick] (0,0) arc [radius=15, start angle=90, end angle=85];
\draw [dashed, thick] (0,0) arc [radius=15, start angle=90, end angle=95];

\draw [fill] (-1.3,-0.06) circle [radius=0.025];
\draw [fill] (1.3,-0.06) circle [radius=0.025];
\draw [fill] (-2.2,2.87) circle [radius=0.025];
\draw [fill] (2.2,2.87) circle [radius=0.025];
\draw [fill] (-1.6,0.92) circle [radius=0.025];
\draw [fill] (1.6,0.92) circle [radius=0.025];

\draw [thick, dotted] (-2.2,2.87)--(-1.3,-0.06);
\draw [thick, dotted] (2.2,2.87)--(1.3,-0.06);

\draw [thick] (-2.2,2.87)--(-1.6,0.92);
\draw [thick] (2.2,2.87)--(1.6,0.92);

\draw [thick] (-2.2,2.87) to [out=335, in=205] (2.2,2.87);
\draw [thick] (-1.6,0.92) to [out=348, in=192] (1.6,0.92);

\node [above] at (-2.2,2.87) {\scriptsize$w$};
\node [above] at (2.2,2.87) {\scriptsize$z$};
\node [below] at (-1.75,0.92) {\scriptsize$x$};
\node [below] at (1.75,0.92) {\scriptsize$y$};
\node [below] at (-1.3,-0.06) {\scriptsize$x_0$};
\node [below] at (1.3,-0.06) {\scriptsize$y_0$};

\node [right] at (5.48,2.15) {\scriptsize $H_{8\delta-T}$};
\node [right] at (3.88,0.52) {\scriptsize $H_{2\delta-T}$};
\node [right] at (2.8,-0.28) {\scriptsize $H_{-T}$};

\end{tikzpicture}
\caption{\label{ABCD}}
\end{figure}

Hence we have constructed a complete visibility surface $S$ with bounded nonpositive curvature and finite volume whose universal cover is not Gromov hyperbolic. Examples in higher dimensions can be obtained similarly using warped products.

\bibliographystyle{amsalpha}
\bibliography{Margulis} 

\providecommand{\bysame}{\leavevmode\hbox to3em{\hrulefill}\thinspace}
\providecommand{\MR}{\relax\ifhmode\unskip\space\fi MR }
\providecommand{\MRhref}[2]{%
  \href{http://www.ams.org/mathscinet-getitem?mr=#1}{#2}
}
\providecommand{\href}[2]{#2}
\begin{thebibliography}{BGT12}

\bibitem[AB98]{AB98}
Scot Adams and Werner Ballmann, \emph{Amenable isometry groups of {H}adamard
  spaces}, Math. Ann. \textbf{312} (1998), no.~1, 183--195.

\bibitem[AS92]{AS}
Uwe Abresch and Viktor Schroeder, \emph{Graph manifolds, ends of negatively
  curved spaces and the hyperbolic {$120$}-cell space}, J. Differential Geom.
  \textbf{35} (1992), no.~2, 299--336.

\bibitem[Bal01]{Ball-book}
W.~Ballmann, \emph{Spaces of nonpositive curvature}, Jahresber. Deutsch.
  Math.-Verein. \textbf{103} (2001), no.~2, 52--65. \MR{1851331}

\bibitem[Bel16]{Bel-survey}
Igor Belegradek, \emph{Topology of open nonpositively curved manifolds},
  Geometry, topology, and dynamics in negative curvature, London Math. Soc.
  Lecture Note Ser., vol. 425, Cambridge Univ. Press, Cambridge, 2016,
  pp.~32--83.

\bibitem[BF21]{MR4275871}
Emmanuel Breuillard and Koji Fujiwara, \emph{On the joint spectral radius for
  isometries of non-positively curved spaces and uniform growth}, Ann. Inst.
  Fourier (Grenoble) \textbf{71} (2021), no.~1, 317--391. \MR{4275871}

\bibitem[BGS85]{BGS}
Werner Ballmann, Mikhael Gromov, and Viktor Schroeder, \emph{Manifolds of
  nonpositive curvature}, Progress in Mathematics, vol.~61, Birkh\"auser Boston
  Inc., Boston, MA, 1985.

\bibitem[BGT12]{BGT-gap}
Emmanuel Breuillard, Ben Green, and Terence Tao, \emph{The structure of
  approximate groups}, Publ. Math. Inst. Hautes \'{E}tudes Sci. \textbf{116}
  (2012), 115--221.

\bibitem[BH99]{BH-book}
Martin~R. Bridson and Andr{\'e} Haefliger, \emph{Metric spaces of non-positive
  curvature}, Grundlehren der Mathematischen Wissenschaften [Fundamental
  Principles of Mathematical Sciences], vol. 319, Springer-Verlag, Berlin,
  1999.

\bibitem[BK81]{BK81}
Peter Buser and Hermann Karcher, \emph{Gromov's almost flat manifolds},
  Ast\'erisque, vol.~81, Soci\'et\'e Math\'ematique de France, Paris, 1981.

\bibitem[BO69]{BO-wp}
R.~L. Bishop and B.~O'Neill, \emph{Manifolds of negative curvature}, Trans.
  Amer. Math. Soc. \textbf{145} (1969), 1--49.

\bibitem[Bow93]{Bow93}
B.~H. Bowditch, \emph{Discrete parabolic groups}, J. Differential Geom.
  \textbf{38} (1993), no.~3, 559--583.

\bibitem[Cap09]{Caprace09}
Pierre-Emmanuel Caprace, \emph{Amenable groups and {H}adamard spaces with a
  totally disconnected isometry group}, Comment. Math. Helv. \textbf{84}
  (2009), no.~2, 437--455.

\bibitem[CC96]{CC-ml}
Jeff Cheeger and Tobias~H. Colding, \emph{Lower bounds on {R}icci curvature and
  the almost rigidity of warped products}, Ann. of Math. (2) \textbf{144}
  (1996), no.~1, 189--237.

\bibitem[CM13]{CM13}
Pierre-Emmanuel Caprace and Nicolas Monod, \emph{Fixed points and amenability
  in non-positive curvature}, Math. Ann. \textbf{356} (2013), no.~4,
  1303--1337.

\bibitem[CRX18]{CRX-gap}
L.~{Chen}, X.~{Rong}, and S.~{Xu}, \emph{{A Geometric Approach to the Modified
  Milnor Problem}}, ArXiv e-prints (2018).

\bibitem[DC09]{Cerb09}
Luca~Fabrizio Di~Cerbo, \emph{A gap property for the growth of closed
  3-manifold groups}, Geom. Dedicata \textbf{143} (2009), 193--199.

\bibitem[dlH00]{Harpe-book}
Pierre de~la Harpe, \emph{Topics in geometric group theory}, Chicago Lectures
  in Mathematics, University of Chicago Press, Chicago, IL, 2000.

\bibitem[Ebe80]{Eberlein80}
Patrick Eberlein, \emph{Lattices in spaces of nonpositive curvature}, Ann. of
  Math. (2) \textbf{111} (1980), no.~3, 435--476.

\bibitem[Ebe96]{Eber-book}
Patrick~B. Eberlein, \emph{Geometry of nonpositively curved manifolds}, Chicago
  Lectures in Mathematics, University of Chicago Press, Chicago, IL, 1996.

\bibitem[EHS93]{EHS-survey}
Patrick Eberlein, Ursula Hamenst\"{a}dt, and Viktor Schroeder, \emph{Manifolds
  of nonpositive curvature}, Differential geometry: {R}iemannian geometry
  ({L}os {A}ngeles, {CA}, 1990), Proc. Sympos. Pure Math., vol.~54, Amer. Math.
  Soc., Providence, RI, 1993, pp.~179--227.

\bibitem[EO73]{EO-v}
P.~Eberlein and B.~O'Neill, \emph{Visibility manifolds}, Pacific J. Math.
  \textbf{46} (1973), 45--109.

\bibitem[Gri85]{Grig-c}
R.~I. Grigorchuk, \emph{Degrees of growth of {$p$}-groups and torsion-free
  groups}, Mat. Sb. (N.S.) \textbf{126(168)} (1985), no.~2, 194--214.

\bibitem[Gri14]{MR3289926}
Rostislav Grigorchuk, \emph{Milnor's problem on the growth of groups and its
  consequences}, Frontiers in complex dynamics, Princeton Math. Ser., vol.~51,
  Princeton Univ. Press, Princeton, NJ, 2014, pp.~705--773. \MR{3289926}

\bibitem[Gro78]{Gromov81}
M.~Gromov, \emph{Manifolds of negative curvature}, J. Differential Geom.
  \textbf{13} (1978), no.~2, 223--230.

\bibitem[Gro81]{Gromov-pg}
Mikhael Gromov, \emph{Groups of polynomial growth and expanding maps}, Inst.
  Hautes \'{E}tudes Sci. Publ. Math. (1981), no.~53, 53--73.

\bibitem[Gro07]{Gromov-book}
Misha Gromov, \emph{Metric structures for {R}iemannian and non-{R}iemannian
  spaces}, english ed., Modern Birkh\"{a}user Classics, Birkh\"{a}user Boston,
  Inc., Boston, MA, 2007, Based on the 1981 French original, With appendices by
  M. Katz, P. Pansu and S. Semmes, Translated from the French by Sean Michael
  Bates.

\bibitem[Hei76]{Hein76}
Ernst Heintze, \emph{Mannigfaltigkeiten negativer kr\"ummung},
  Habilitation-schrift, Univ. of Bonn (1976).

\bibitem[HIH77]{HH-77}
Ernst Heintze and Hans-Christoph Im~Hof, \emph{Geometry of horospheres}, J.
  Differential Geom. \textbf{12} (1977), no.~4, 481--491 (1978).

\bibitem[KM99]{KM}
Anders Karlsson and Gregory~A. Margulis, \emph{A multiplicative ergodic theorem
  and nonpositively curved spaces}, Comm. Math. Phys. \textbf{208} (1999),
  no.~1, 107--123.

\bibitem[KW11]{KW-ml}
V.~{Kapovitch} and B.~{Wilking}, \emph{{Structure of fundamental groups of
  manifolds with Ricci curvature bounded below}}, ArXiv e-prints (2011).

\bibitem[Lee00]{MR1934160}
Bernhard Leeb, \emph{A characterization of irreducible symmetric spaces and
  {E}uclidean buildings of higher rank by their asymptotic geometry}, Bonner
  Mathematische Schriften [Bonn Mathematical Publications], vol. 326,
  Universit\"{a}t Bonn, Mathematisches Institut, Bonn, 2000. \MR{1934160}

\bibitem[Rag72]{Rag72}
M.~S. Raghunathan, \emph{Discrete subgroups of {L}ie groups}, Springer-Verlag,
  New York-Heidelberg, 1972, Ergebnisse der Mathematik und ihrer Grenzgebiete,
  Band 68.

\bibitem[Sch84]{Sch84}
Viktor Schroeder, \emph{Finite volume and fundamental group on manifolds of
  negative curvature}, J. Differential Geom. \textbf{20} (1984), no.~1,
  175--183.

\bibitem[Wu18]{Wu-para}
Yunhui Wu, \emph{Translation lengths of parabolic isometries of {${\rm
  CAT}(0)$} spaces and their applications}, J. Geom. Anal. \textbf{28} (2018),
  no.~1, 375--392.

\end{thebibliography}
\end{document}